\RequirePackage{fix-cm} 
\documentclass[a4paper,twoside,12pt,reqno]{amsart}

\usepackage{fixltx2e}     

\usepackage[english]{babel}
\usepackage[latin1]{inputenc}

\usepackage{indentfirst,verbatim}

\usepackage{amsmath,amsfonts,amssymb,amsgen,amsbsy,eucal,mathrsfs,dsfont}
\usepackage{stmaryrd}

\usepackage[square,numbers]{natbib}
\usepackage{a4wide,verbatim}

\usepackage{epsfig,rotating,color}
\usepackage{multicol}
\usepackage{tikz,pgfplots}
\usepackage{graphicx}
\usepackage{mathtools}
\usepackage{textcomp}

\usepackage{caption}
\usepackage{nicefrac}

\usepackage{hyperref}

\usepackage[square,numbers]{natbib}
\usepackage{systeme}

\hypersetup{
colorlinks=true,  
breaklinks=true,  
urlcolor= blue, 
linkcolor= blue, 
bookmarksopen=true,  
pdftitle={},  
pdfauthor={}, 
pdfsubject={},  
citecolor=red,  
}

\newcommand*{\bydef}{\overset{\rm def}{=}}
\newcommand*{\norm}[1]{\left\Vert #1\right\Vert}

\newcommand*{\id}{\operatorname{Id}}

\newcommand*{\loc}{\mathrm{loc}}

\newcommand{\R}{\mathbb{R}}
\newcommand{\T}{\mathbb{T}}
\newcommand{\N}{\mathbb{N}}
\newcommand{\Z}{\mathbb{Z}}

\newcommand{\ud}{\mathrm{d}}

\usepackage{amsthm}
\newtheorem{theorem}{Theorem}[section]

\newtheorem{lemma}[theorem]{Lemma}
\newtheorem{proposition}[theorem]{Proposition}

\theoremstyle{remark}
\newtheorem{remark}{Remark}

\usepackage{enumitem}
\newlist{steps}{enumerate}{1}
\setlist[steps, 1]{label = Step \arabic*:}

\usepackage[textsize=tiny,textwidth=27mm]{todonotes}

\newcommand{\eqdef}{\stackrel{\text{\tiny{def}}}{=}}

\title[Euler--Poisson]{Singular traveling waves for the Euler--Poisson system}     

\author[]{Billel Guelmame, Taoufik Hmidi, Haroune Houamed and Fr\'ed\'eric Rousset}

\newcommand{\nfont}{\fontshape{n}\selectfont}

\address{({\nfont\textbf{Billel Guelmame}})  New York University Abu Dhabi, Abu Dhabi, United Arab Emirates.} 
\email{billel.guelmame@nyu.edu}

\address{({\nfont\textbf{Taoufik Hmidi}}) New York University Abu Dhabi, Abu Dhabi, United Arab Emirates.} 
\email{th2644@nyu.edu}

\address{({\nfont\textbf{Haroune Houamed}}) Ko\c c University, Istanbul, Turkey; and New York University Abu Dhabi, Abu Dhabi, United Arab Emirates.} 
\email{hhouamed@ku.edu.tr; haroune.houamed@nyu.edu}

\address{({\nfont\textbf{Fr\'ed\'eric Rousset}}) Universit\'e Paris-Saclay, Paris, France.} 
\email{frederic.rousset@math.u-psud.fr}

\let\oldtocsection=\tocsection
 
\let\oldtocsubsection=\tocsubsection

\renewcommand{\tocsection}[2]{\hspace{0em}\oldtocsection {#1}{#2}}
\renewcommand{\tocsubsection}[2]{\hspace{2em}\oldtocsubsection{#1}{#2}}

\numberwithin{equation}{section}

\usepackage{tabularx}
\usepackage[pagewise]{lineno} %

\begin{document}

\begin{abstract}
We consider the Euler--Poisson system for ions where the electrons are given by a Maxwell--Boltzmann distribution, and we investigate the existence of one-dimensional periodic traveling waves. More precisely, we first establish the existence of a smooth global  
branch of bifurcation emanating from a constant equilibrium. We then construct a singular traveling wave 
emerging as the limiting profile at the end of the global curve of bifurcation. 

Our analysis accommodates a wide class of pressure laws and provides a comprehensive characterization of both smooth and singular traveling waves. A central difficulty in this model arises from the exponential nonlinearity, induced by the nonlocal Poisson--Boltzmann equation, which prevents any explicit representation of the electron field in terms of the ion density. This poses significant obstacles compared to previous studies on related models, where such explicit formulas were crucial for global bifurcation arguments.  
\end{abstract}

\maketitle

\tableofcontents

\section{Introduction}


The study of nonlinear wave phenomena in dispersive media is a central topic in mathematical physics, with important applications ranging from fluid dynamics to plasma physics. Among the fundamental models arising in this context, the Euler--Poisson system occupies a prominent place as a reduced description of ion dynamics in a plasma, where the electrostatic interaction is mediated through a self-consistent field. 

In the present work, we investigate the existence and qualitative properties of periodic traveling wave solutions to a one-dimensional Euler--Poisson system with Boltzmann distributed electrons. More precisely, we consider a regime in which the electrons are assumed to have reached a thermodynamical equilibrium given by the Maxwell--Boltzmann distribution $ \rho_{e}= \mathrm{e}^\phi$, leading to a nonlinear Poisson equation of Maxwell--Boltzmann type. This setting provides a physically relevant and mathematically rich framework, in which nonlinear transport, pressure effects, and nonlocal interactions are strongly coupled. 
In more precise terms, the actual model of our interest here is given by 
\begin{equation}
	\label{EP}\tag{EP}
	\begin{cases}
		\partial_t \rho + \partial_x(u\rho) =0,
		\\
		\partial_t u + \tfrac{1}{2} \partial_x(u^2) + \partial_x \big (p(\rho)\big ) = - \partial_x \phi , 
		\qquad \qquad (t,x) \in \mathbb{R}^+\times \mathbb{T},
		\\
		-\partial_x^2  \phi + \mathrm{e}^{\phi} = \rho,
	\end{cases}
\end{equation}
which, again, describes the dynamics of ions in a plasma. Here, $\rho>0$ stands for the density of ions, $u$ is the velocity of ions and $\partial_{x} \phi$
 represents the electric field generated by the motion of the  particles. The pressure  $P(\rho)$, on the other hand, 
 is given  by 
 \begin{equation}
\label{Pp}
P'(\rho) \bydef  \rho p'(\rho), 	
\end{equation}
for some profile $p$, which will be discussed in a subsequent section.
 
This system of equations  can be justified as an asymptotic limit from the full  two-fluid model in the quasi-neutral and zero electron
mass limits, we refer for example to \cite{Alves-Tzavaras, Grenier-Guo-Pausader}. It  supports various classes of solitary wave solutions, as well as the two fluid models in some regimes.  We refer for example to \cite{Bae-Kwon, Degond-Cordier} for more details.
Here, we are interested  in  the one-dimensional case with periodic boundary conditions, where we denote  $\mathbb T \bydef \mathbb R/L \mathbb Z$  the one-dimensional torus with period $L>0$. We aim at investigating the existence and  qualitative properties  of smooth and singular periodic traveling  waves for \eqref{EP}. 

This type of solutions plays a fundamental role in understanding the nonlinear dynamics of dispersive systems, including \eqref{EP}. In particular, periodic traveling waves often arise as coherent structures and may serve as building blocks for more complex dynamics. Their analysis has attracted considerable attention in various contexts, including water waves, plasma models, and kinetic equations. A classical example is the Whitham equation \cite{Ehrnstrom3, Ehrnstrom2,  Ehrnstrom5}, which incorporates full dispersion effects and has been extensively studied in connection with the formation of highest waves and singular profiles. Similarly, the Vlasov--Poisson system provides a kinetic description of plasmas, where nonlocal effects also play a crucial role.

Despite earlier developments, the Euler--Poisson system considered here presents several distinctive features that make its analysis substantially more challenging. First, after reduction to a traveling wave ansatz, the problem leads to a nonlinear and nonlocal equation in which the nonlocality is governed by a \emph{nonlinear} operator. This contrasts sharply with many previously studied models, where the nonlocal term is typically linear. Second, the local nonlinear structure may become singular in regimes where the density approaches zero, introducing additional analytical difficulties. These aspects prevent a direct adaptation of existing techniques and require the development of new ideas.

\subsection{Reformulation of the problem}

Let us now shed   light on the reduction of the problem, which is done by applying the traveling wave ansatz and performing some useful reformulations. To this end, we first observe that \eqref{EP} is Galilean invariant, in the sense that if $(u,\rho,\phi)$ is a solution to \eqref{EP}, then for any $\omega \in \mathbb{R}$, the triplet
 \begin{equation*}
  \rho  (t,x-\omega t),  \quad 	u(t,x-\omega t) +\omega,     \quad \phi (t,x-\omega t)
 \end{equation*}
 is also a solution to \eqref{EP}, for any $\omega\in \mathbb R$. Given this, without loss of generality, we are going to look for stationary waves around the constant solution   
 \begin{equation*}
( \rho,  u,  \phi) = (1, c,0), \quad   c\in \mathbb R. 	
 \end{equation*} 
 Generally speaking, any stationary solution $(\rho,u,\phi)$ of \eqref{EP} is governed by the system 
 \begin{equation*}
	\begin{cases}
		u\rho =k_1,
		\\
		 \tfrac{1}{2} u^2 + p(\rho) = -  \phi + k_2 , 
		 
		\\
		-\partial_x^2  \phi + \mathrm{e}^{\phi} = \rho,
	\end{cases}
\end{equation*}
for some constants $k_1,k_2\in \mathbb R$. This can be recast as
\begin{equation*} 
	\begin{cases}
		\tfrac{1}{2} \left (\tfrac{k_1}{\rho} \right)^2 + p(\rho) = -\phi + k_2,		 
		\\
		-\partial_x^2  \phi + \mathrm{e}^{\phi} = \rho,
		\\
		u =k_1/\rho,
	\end{cases}
\end{equation*}
where, as we shall see later on, the first equation is the main one that we will be  focusing on. Recalling that we are looking for solutions around $ (1, c,0)$,   we   further adopt the condition 
\begin{equation*}
	k_1 =c , \quad k_2= p(1) + \tfrac{1}{2} c^2. 
\end{equation*}
Thus, we arrive at the final reformulation of our problem, which consists of  finding a solution $(c, \rho) $ of the problem 
\begin{equation}\label{main_EQ}
     \tfrac{c^2}{2} \left (\tfrac{1}{\rho^2} -1\right) + p(\rho)-p(1) + \mathcal H^{-1}(\rho) =0,
\end{equation}
where  $\mathcal H^{-1}$ denotes the inverse operator of 
\begin{equation}\label{Hdef}
	\mathcal H(\phi ) \bydef  - \partial_x^2  \phi + \mathrm{e}^{\phi }.
\end{equation}
The precise sense in which the operator $\mathcal H$ and its inverse are well defined will be discussed in a subsequent section, later on.

 \subsection{Choice of the pressure}\label{section:pressure}

Throughout the paper,  we will only be working with the profile $p$, representing the pressure through the equation \eqref{Pp}. We will assume that $p$ is locally  analytic on $(0,\infty)$. Nevertheless, it is to be emphasized that the analyticity  condition will only be required in the global bifurcation component of our analysis, while a $C^k$ regularity, for some $k\geqslant 2$, will be sufficient to construct (smooth) traveling waves around the equilibrium.
\\
We assume in addition, for any $\xi \in (0,\infty)$, that  
\begin{equation}\label{p_conditions1} 
p'(\xi)>0, \qquad  3 p'(\xi)+\xi p''(\xi) > 0, \qquad	\lim_{\xi \to 0^+ } \xi^3 p'(\xi) =0, \qquad \lim_{\xi \to \infty } W(\xi)=\infty,
\end{equation} 
where we set
\begin{equation*}
	W(\xi) \eqdef \frac{\xi^4 p'(\xi) - 2 \xi (p(\xi)-p(1)) - 2 \xi \log(\xi)}{ \max_{\eta \in [1 , \xi]} p'(\eta)+1 }, \qquad \text{for all }\xi \in [1,\infty) .
\end{equation*}
We defer  the discussion of these assumptions to the remark at the end of this section.
Note that the condition on the limit of $W(\xi)$ as $\xi \to \infty$ implies that 
\begin{equation*}
	\lim_{\xi \to \infty} \xi \left(\xi^3 p'(\xi)  - 2  (p(\xi)-p(1))   -2 \log(\xi) \right) =\infty,
\end{equation*}
and, therefore, that  
\begin{equation}\label{x3p' limit}
		\lim_{\xi \to \infty}  \xi^3 p'(\xi) =\infty.
\end{equation}
We deduce from \eqref{p_conditions1} and \eqref{x3p' limit} that the map 
\begin{equation*}
	\xi \mapsto \xi^3 p'(\xi)
\end{equation*}
is increasing and a bijection from $(0,\infty)$ into itself. 
\\
Let us now introduce some notation that will useful  later on. For any $c \in \R \setminus \{0\}$, we denote $a^*=a^*(c)>0$ the unique solution to the equation (in $\xi$)
\begin{equation}\label{a*def}
	\xi^3 p'(\xi)=c^2, \quad \text{  } \xi >0. 
\end{equation}
Note, moreover, that the map $c\mapsto a^*(c)$ is increasing on $(0,\infty)$ and 
\begin{equation}\label{a^*lim}
	\lim_{c \to \infty} a^*(c)=\infty.
\end{equation}
Furthermore, we introduce the  function 
\begin{equation}\label{G_c:def}
	\mathcal G_c(\xi) \bydef  \tfrac{c^2}{2} \left (\tfrac{1}{\xi^2} -1\right) + p(\xi)-p(1),
\end{equation} 
which represents the local component of the functional of our interest from \eqref{main_EQ}.
It is readily seen that this function  is analytic on $(0,\infty)$. Moreover, it holds that 
\begin{equation*}
	\mathcal G_c'(a^*(c))=0,
\end{equation*}
while
\begin{equation*}
	\mathcal G_c'(\xi)<0,  
\end{equation*} 
for any $\xi \in (0,a^*(c))$. In addition, the function $\mathcal G_c$ satisfies 
\begin{equation*}
	\mathcal G_c''(\xi)>0,  
\end{equation*} 
for any $\xi \in (0,a^*(c)]$. This convexity condition follows from the definition of $a^*(c)$ in \eqref{a*def} and the monotonicity of $\xi \mapsto \xi^3 p'(\xi)$ with \eqref{p_conditions1}.

\begin{remark}(About the assumptions on the pressure)\label{rem:pressure}
\begin{itemize}
	\item 
Let $\gamma>0$ and $\kappa>0$.
 In view of \eqref{Pp}, note that the example of pressures given by  
\begin{equation*}
	P_1(\rho) =   \kappa \rho^\gamma \qquad \text{and} \qquad  P_2(\rho) = \kappa \log(\rho)
\end{equation*}
which respectively correspond to 
\begin{equation*}
	p_1(\rho)= 
	\begin{cases}
		\tfrac{\gamma \kappa }{\gamma-1} \rho^{\gamma-1}, &\gamma \neq 1,\\
		\kappa \log(\rho), &\gamma =1
	\end{cases} 
	\qquad \text{and} \qquad  p_2(\rho) = \tfrac{-\kappa}{\rho}
\end{equation*}
are admissible and enjoy the condition  \eqref{p_conditions1}.

\item Note that if \eqref{p_conditions1} is satisfied, then 
\begin{equation*}
	\lim_{\xi \to \infty } W_\delta(\xi)=\infty,
\end{equation*}
for any $\delta \in (0,\infty)$, where 
\begin{equation}\label{Wdef}
		W_\delta(\xi) \eqdef \frac{\xi^4 p'(\xi) - 2 \xi (p(\xi)-p(1)) - 2 \xi \log(\xi)}{ \max_{\eta \in [\delta , \xi]} p'(\eta)+1 }, \qquad \xi \in [\delta,\infty).
\end{equation}
This will be used in Proposition \ref{prop:bound:c} below with $\delta \in (0,1)$. 
 
\item  To conclude this discussion, let us point out that the first three conditions in \eqref{p_conditions1} with \eqref{x3p' limit} are essentially considered to guarantee some important properties of the function $\mathcal G_c$ introduced above. Meanwhile, it is to be emphasized that the condition on the limit of $W(\xi)$ in \eqref{p_conditions1} will solely  be employed in the proof of Proposition \ref{prop:bound:c} to ensure  that any solution $ (c,\rho )$ of \eqref{main_EQ} is identically equal to the trivial solution $(c,1)$ if the value of $c$ is too large.\\ 

\end{itemize}
\end{remark}

 \subsection{Main results}

 We are now in a position to state our  contribution, which is outlined below in two separate  theorems for the sake of clarity. The first one establishes the existence of a global bifurcation curve of (smooth) solutions $ (c,\rho)$ to the equation \eqref{main_EQ}. This consists of a whole branch of traveling wave  solutions of the original Euler--Poisson system \eqref{EP}.

 \begin{theorem}[Smooth global branch of bifurcation]\label{thm:BifCurv}
Let $p$ be a locally  analytic function on $(0,\infty)$ that satisfies \eqref{p_conditions1}. Then, there exists a  set $\mathcal N \subset (0,\infty)$ of cardinal at most $8$, such that, for any $L \in (0,\infty) \setminus \mathcal N$,
	there exist  $\delta_0\in (0,1)$, a family $(c_s, \rho_s)_{s\in [0,\infty)} \subset \R \times C^\infty (\mathbb T)$ of solutions to \eqref{main_EQ} satisfying the following:
 	\begin{enumerate}[label=(\roman*)]
 	   \item Symmetries and monotonicity: $ \rho_s $ is $L$-periodic and even. Moreover, for any $s>0$, the function $\rho_s$ is increasing on $(-\nicefrac L2,0)$.
 	   \item Boundedness: For any $s \geqslant 0$, it holds that  
 	   \begin{equation*}
 	   	\delta_0<  \rho_s(x) <a^*(c_s), \quad \text{for all } x\in \mathbb T,
 	   \end{equation*}
 		  where $a^*(\cdot )$ is defined as the unique solution of \eqref{a*def} 
 		\item Bifurcation point: at $s=0$, we have that
 		  \begin{equation*}
 		  	(c_0, \rho_0)= \left (\pm \sqrt{p'(1) + \tfrac{1}{1+ (\nicefrac {2\pi}{L})^2}}, 1 \right).
 		  \end{equation*}
 		  \item Local analyticity: For any $\tilde{s} \in (0,\infty)$, there exists a continuous and injective parameterization $\varphi :(-1,1) \to \R$ such that $\varphi(0)=\tilde{s}$ and the map $t \mapsto (c_{\varphi(t)},\rho_{\varphi(t)})$ is analytic.
 		   	\end{enumerate}
 \end{theorem} 

	  \begin{remark}\label{remark:about:N} It is to be emphasized later in our analysis (see Proposition \ref{prop:Psi} and its proof) that the condition $L \notin \mathcal{N}$, for some a priori ``potentially empty'' set $\mathcal{N}$, is used to ensure that the local bifurcation curve is not of the form $(c_{s},1)_{s \in (-\varepsilon,\varepsilon)}$, a property that appears to be required in the global bifurcation theorem. 

We conjecture that $\mathcal{N}$ is in fact empty. We actually can show this is the case of many pressure laws, including $p(\rho)=\rho ^2$ (see Remark \ref{remark:about:N:2}, below).  However, confirming  this claim in the general case would likely require lengthy and technical computations, which go  beyond the scope of the present paper. For this reason, we do not pursue this issue further here.  
\end{remark}

In our second theorem, below, we establish the existence of a singular traveling wave solution of \eqref{EP}, which  appears to be a limiting profile at the end of the global branch of solutions constructed in Theorem \ref{thm:BifCurv}.

\begin{theorem}[Singular traveling-wave]\label{thm:LimitObject}

Up to the extraction of a subsequence, the solution $(c_s,\rho_s)$ of \eqref{main_EQ} given by Theorem~\ref{thm:BifCurv} converges, as $s \to \infty$, to some limit $(c_\infty, \rho_\infty)$ in $\R \times C(\T)$.
Moreover, $(c_\infty, \rho_\infty)$ solves \eqref{main_EQ}, with $\rho_\infty \in W^{1,\infty}(\T)$ and  is of class $C^\infty$ on the set  $\mathbb R\setminus \{ kL: k\in \mathbb Z\}$. 
Furthermore, this limiting profile satisfies   
\begin{gather*}
 		  	\delta_0<  \rho_\infty(x) <a^*(c_\infty), \quad \text{for all } x\in \mathbb R\setminus \{ kL : k\in \mathbb Z\}, \\
 		  	  \rho_\infty(x) =a^*(c_\infty), \quad \text{for all } x\in \{ kL: k\in \mathbb Z\}.
\end{gather*}
 		  Additionally, $ \rho_\infty$  develops corner singularities at its maxima. More precisely,  one has the asymptotic behavior 
 		  \begin{equation*}
	a^*(c_\infty)-\rho_\infty(x)= \theta |x|+ \mathcal{O}(x^2),
		  \end{equation*}
for any $|x|\ll 1$, where 
\begin{equation*}
	\theta \bydef \sqrt{\tfrac{-\partial_x^2\mathcal{H}^{-1}(\rho_\infty) (0)}{\mathcal G_{c_\infty}''(a^*(c_\infty))}}  =  \sqrt{\tfrac{a^*(c_\infty)-\mathrm{e}^{-\mathcal G_{c_\infty}(a^*(c_\infty))}}{\mathcal G_{c_\infty}''(a^*(c_\infty))}} >0.
\end{equation*}
\end{theorem}

Notice that we can restrict our analysis below to the case of $c_0>0$, where $c_0$ is the bifurcation point introduced in Theorem \ref{thm:BifCurv}, for the case of a negative speed $c_0$ can be recovered by a symmetry argument. Indeed, this can be done by virtue of  the observation that replacing $u$ and $x$ in \eqref{EP} by $-u$ and $-x$, respectively, allows us to recover the case of negative speeds $c<0$  from the case of positive ones. 
We will thus   only be considering the case $c_0>0$ in the discussion and the  proofs of our main theorems.

\subsection{Challenges and  connections with relevant works}
In the proofs, we primarily rely on the formulation \eqref{main_EQ}, which takes the form of a nonlinear and nonlocal equation. It is worth noting that related results have been established for the Whitham equation \cite{Ehrnstrom3, Ehrnstrom2,  Ehrnstrom5}, a full-dispersion counterpart of the KdV equation, which can be derived from the water waves system with the same order of accuracy as the KdV equation (see, for instance, \cite{Lannes-book}). The singularity of the highest wave was conjectured by Whitham in connection with the Stokes conjecture for water waves, which was later resolved in \cite{Toland, Plotnikov}.
However, our formulation \eqref{main_EQ} does not fall within the class of equations considered in these works, which typically consist of a linear nonlocal operator combined with a quadratic nonlinearity of the form $u(u-\mu)$, for some $\mu>0$. In contrast, in our setting, the nonlocal component is governed by a nonlinear operator, while the local nonlinear term may become singular as $\rho \to 0$. This structural difference constitutes a major source of new analytical difficulties.
More precisely, one of the central challenges in proving the main results lies in the delicate analysis required to understand fine properties of the inverse operator $\mathcal{H}^{-1}$ associated with \eqref{Hdef}. Although several results concerning this operator are already available (see Section \ref{elliptic_problem} for a detailed comparison with the contributions of the present work), their application within the global bifurcation framework demands further refinement. Compared to related models---such as the Whitham equation \cite{Ehrnstrom4, Ehrnstrom1, Ehrnstrom3, Ehrnstrom2,  Ehrnstrom5} or the Vlasov--Poisson system \cite{R26}---the main additional difficulty stems from the nonlinear nature of the operator $\mathcal{H}$ and, consequently, of its inverse. This feature prevents a direct adaptation of existing arguments developed for those models to the present setting \eqref{EP}.

A key idea underlying our approach to overcome these difficulties is to introduce a suitable linearization scheme, which connects the nonlinear operator $\mathcal{H}^{-1}$ to an associated linearized operator that retains the essential structural properties. Once these properties are established at the linearized level, they can be transferred to the original operator $\mathcal{H}^{-1}$ through relatively elementary arguments. We refer to  Sections \ref{section:nonlin:OP} and \ref{elliptic_problem_lin}   for the detailed analysis of the nonlinear operator $\mathcal{H}^{-1}$ and its linearized counterpart, respectively. The way in which this connection is exploited in the context of traveling waves will become clear in the proofs presented in Section \ref{section:preleminary:traveling_waves}, thereafter.

\subsection{Organization of the paper}

Our analysis is divided into two parts. The first part is devoted to a priori estimates and preparatory results, which are developed in Sections \ref{section:preliminaries} and \ref{section:preleminary:traveling_waves}. The second part contains the core of the paper, where we establish the existence of small- and large-amplitude traveling wave solutions to \eqref{EP}, including both smooth and singular profiles; this is carried out in Sections \ref{section:local:theory} and \ref{sec:global} respectively.

Finally, for the sake of completeness, we collect in an appendix the main results from bifurcation theory that are used in the proofs of our principal theorems.

\section{Preliminaries}\label{section:preliminaries} 

This section is devoted to the proof of some elementary lemmas which will be useful in our analysis, later on. One of the core findings in this section will be related to the study the operator $\mathcal H$ defined in \eqref{Hdef} and its inverse. It is to be emphasized that a fine understanding of the inverse of the linear elliptic operator $\lambda \mathrm{Id}-\partial_x^2$ that will also be central in the analysis of the inverse of the nonlinear operator $\mathcal H$, and eventually the traveling waves of the Euler--Poisson equation \eqref{EP}.


\subsection{Some elementary lemmas}
 
For the sake of completeness, we embark with two elementary lemmas on real-valued functions enjoying some specific properties. These lemmas will be used in a subsequent section on the analysis of traveling wave solutions to Euler--Poisson system \eqref{EP}.

\begin{lemma}\label{lemma:MVT}
	Let $a<b$ be two real numbers, and let $h: [a,b] \to \mathbb R$ be a $C^2$ function such that $h'$ has a constant sign on $[a,b]$ and $h''$ does not vanish on $[a,b]$. Then, it holds that 
	\begin{equation*}
		|h(x)-h(y)|\geqslant M_1 |x-y|^2,
	\end{equation*}
	for any $x,y\in [a,b]$, where 
	\begin{equation*}
		M_1 \bydef \tfrac{1}{2} \min_{z\in [a,b]} |h''(z)|.
	\end{equation*}
	Moreover, if $c \in [a,b]$, such that 
	\begin{equation*}
		h'(c)=0,
	\end{equation*}
	then, it holds that 
	\begin{equation*}
		M_1 |x-c|^2 \leqslant |h(x)-h(c)| \leqslant M_2 |x-c|^2,
	\end{equation*}
	for any $x\in [a,b]$, where 
	\begin{equation*}
		M_2 \bydef \tfrac{1}{2} \max_{z\in [a,b]} |h''(z)|.
	\end{equation*}
\end{lemma}

\begin{proof}
	Notice that we can assume, without loss of generality, that 
	\begin{equation}\label{case:***}
		h'(x) \leqslant 0 \qquad \text{ and } \qquad h''(x)> 0 , \quad \text{ for all } x\in [a,b].
	\end{equation}
	Indeed, the complete claimed result, in the more general case   stated in the lamma, can be achieved by replacing $h(x)$ with $ h( - x)$,  $-h(-x) $ or $-h(x)$, which allows us to cover all the possible scenarios corresponding to the different possible signs of $h'$ and $h''$.
\\	
Now, let $x,y\in [a,b]$ with $x<y$. It then follows by the Mean Value Theorem that 
	\begin{equation}\label{MVI}
		|h(x)-h(y)|= h(x)-h(y)= h'(y)(x-y) + \tfrac{1}{2} h''(z)|x-y|^2,
	\end{equation}
	for some $z\in [a,b]$. Thus, we deduce, by employing \eqref{case:***}, that 
	\begin{equation*}
		|h(x)-h(y)|\geqslant  \tfrac{1}{2} h''(z)|x-y|^2 \geqslant M_1 |x-y|^2 ,
	\end{equation*}
	where $M_1$ is defined in the statement of the lemma. We emphasize that the case $x\geqslant y$ can be treated in the exact same way, thereby completing the proof of the desired lower bound. 
	\\
	The upper bound in the case when $h'(c)=0$ is a direct consequence of  \eqref{MVI}.\\
	This completes the proof of the lemma.
\end{proof}

\begin{lemma}\label{lemma:stability:monotonicity}
	Let $f$ be a $C^2$ function  on a given closed interval $[a,b]$. Assume further that 
	\begin{equation*}
		f' > 0 \quad \text{ on } (a,b), \qquad \text{and }  \qquad f''(b)<0<f''(a).
	\end{equation*} 
	Then, any function $g$ in a sufficiently small $C^2$-neighborhood of $f$ satisfying
\begin{equation*}
	g'(a)\geqslant 0 \qquad \text{and} \qquad g'(b) \geqslant 0 ,
\end{equation*}	
is increasing on the interval $(a,b)$, as well.
\end{lemma}

\begin{proof} 
Since $f\in C^2$, there exists $\varepsilon>0$ such that
\begin{equation*}
	f'' \geqslant \tfrac12 f''(a)>0 \quad \text{on } [a,a+\varepsilon]\qquad \text{and} \qquad
	f'' \leqslant \tfrac12 f''(b)<0 \quad \text{on } [b-\varepsilon,b].
\end{equation*}
Let $\mathcal{N}=B(f,r)$ be the closed ball in $C^2$ centered at $f$ with radius  
\begin{equation*}
r \bydef  \tfrac14 \min \left\{  f''(a), - f''(b), \min_{x\in[a+\varepsilon,b-\varepsilon]} f'(x) \right\}>0.
\end{equation*}
Now, let $g $ be fixed in $ \mathcal{N}$ such that $g'(a)\geqslant 0$ and $g'(b)\geqslant 0$. It is then readily seen that 
\begin{equation*}
	g''(x) \geqslant f''(x) - r > f''(x) - \tfrac12 f''(a) \geqslant 0,
\end{equation*}
for any  $x \in [a,a+\varepsilon]$.
Therefore $g'$ is increasing on $[a,a+\varepsilon]$.
Since $g'(a) \geqslant 0$, we have $g'>0$ on $(a,a+\varepsilon]$. 
Similarly, we can show that  $g'>0$ on $[b-\varepsilon,b)$.

On the other hand,  we have that
\begin{equation*}
	g'(x) \geqslant f'(x) - r > f'(x) - \tfrac12 \min_{x\in[a+\varepsilon,b-\varepsilon]} f'(x) > 0,
\end{equation*}
for any $x \in [a+\varepsilon,b-\varepsilon]$.
All in all, we have shown that $g'>0$ on $(a,b)$, thereby  completing the proof of the lemma. 
\end{proof}

\subsection{On a nonlinear elliptic operator}\label{section:nonlin:OP}

This section is exclusively devoted to the analysis of the elliptic equation 
\begin{equation}\label{elliptic_problem}
	\mathcal H(\phi)\bydef -\partial_x^2 \phi + \mathrm{e}^{\phi} =  f,
\end{equation}
which appears in \eqref{Hdef}, where $f$ is a given continuous function 
$$f : \mathbb{T} \mapsto (0,\infty). $$ 
The main result of this section, namely Proposition \ref{prop:elliptic} below, establishes the existence and uniqueness of solutions to \eqref{elliptic_problem}, together with their regularity properties in terms of the source term $f$. We note that similar results were  already  previously established, see   \cite[Section 2.1]{LLC13} and \cite[Section 3]{GM21CPDE} for instance.
For the sake of clarity, we now briefly outline the main differences between our approach and the analysis in \cite{GM21CPDE, LLC13}. The proof in \cite{LLC13} is carried out in the whole space $\mathbb{R}^d$, for $d=1,2,3$, and relies on the method of sub- and supersolutions to solve the equation $\mathcal{H}(\phi)=f$.  On the other hand,  the arguments  in \cite{GM21CPDE} are developed on the $d$-dimensional torus, for any $d\geqslant 1$ (see in particular Section 3.3 therein), and  are based on a variational approach. More precisely,  the solution $\phi $ in \cite{GM21CPDE} is constructed via the decomposition $\overline\phi+ \widehat \phi $, where $\overline\phi$ is the solution of the Poisson equation with the  source term $f-1$, while   $\widehat \phi$ solves a modified  version of the original elliptic problem, depending on $\overline \phi$ and involving the constant source term $1$. The existence of $\widehat{\phi}$ is then obtained by minimizing the energy functional associated with the Euler--Lagrange equation satisfied by $\widehat{\phi}$. \\
In contrast, we present a different approach on the torus, which can also be adapted to the whole space with minor modifications. Our construction of the solution $\phi$ proceeds in two steps. First, we establish existence and uniqueness for a suitable linearized problem, which is treated directly. Then, building on this linear analysis, we apply a fixed point argument (due to Schauder) to solve the original nonlinear problem \eqref{elliptic_problem}. \\
The next proposition formulates the main result of this section, stating the existence of a unique solution to \eqref{elliptic_problem} under some suitable conditions on $f$.

\begin{proposition}\label{prop:elliptic} For  any  given function $f$ in $ C (\mathbb{T})$ satisfying  
\begin{equation*}
	 \inf_{x\in \mathbb T} f (x)   \geqslant \delta ,
\end{equation*}
for some $\delta \in (0,1) $,
the equation \eqref{elliptic_problem} admits a unique solution $\phi \in C^2(\T)$.  Moreover, it holds that
\begin{equation*} 
\norm {\phi}_{C(\mathbb T)} \lesssim	\norm {\phi}_{H^1(\mathbb T)} \lesssim_\delta \norm {f-1}_{L^2(\mathbb T)},
\end{equation*} 
where the implicit constant does not depend on $f$. 
\\
In addition, we have the higher-order estimates
\begin{equation*}
	\norm {\phi}_{C^2(\mathbb{T})} \lesssim_{\delta, \|f\|_{C(\T)}} \norm {f-1}_{C (\mathbb{T})}
\end{equation*}
and
\begin{equation*} 
	\norm {\phi}_{H^2(\mathbb T)} \lesssim_{\delta, \|f\|_{C(\T)}} \norm {f-1}_{L^2(\mathbb T)}.
\end{equation*} 
Furthermore, we have that 
\begin{equation*}
	\log \left (\min_{y\in \mathbb T} f(y)\right)  \leqslant   \phi(x)\leqslant \log \left (\max_{y\in \mathbb T} f(y)\right),
\end{equation*}
for   any $x\in \mathbb T$.
\end{proposition}

\begin{proof}
The proof will be performed in  four steps. 
\\
{\em\underline{ Step $1 :$} Uniqueness.}
	The uniqueness part  of the proof  follows by a direct $L^2$-energy estimate. To this end, we invoke the Mean Value Theorem, which yields
	\begin{equation}\label{monotonicity:H}
		\int_{\mathbb T} \left( \mathrm{e}^{\phi_1} - \mathrm{e}^{\phi_2}\right)(\phi_1- \phi_2)\, \ud x \geqslant \mathrm{e}^{\min \{ \inf_x \phi_1(x),\inf_x \phi_2(x) \}}\int_{\mathbb T}  (\phi_1- \phi_2)^2\, \ud x \geqslant 0,
	\end{equation}
	for any continuous functions $\phi_1$ and $\phi_2$. Assuming now that \eqref{elliptic_problem} admits two solutions $\phi_1$ and $\phi_2$, we consider  the equation for their  difference, which takes the form	\begin{equation*}
		-\partial_x^2 \phi_1+\partial_x^2 \phi_2 + \mathrm{e}^{\phi_1} - \mathrm{e}^{\phi_2}=0. 
	\end{equation*}
	Multiplying by $\phi_1-\phi_2$ and integrating on the torus yields, by employing \eqref{monotonicity:H},  that $\phi_1=\phi_2$. 
	\\
{\em\underline{ Step $2 :$}  A companion linearized elliptic problem.}
Let us introduce the nonempty closed, convex and bounded set 
	\begin{equation*}
		S_{a,b} \bydef \{ \phi \in C(\mathbb {T}):  a \leqslant \phi (x)  \leqslant b  ,\quad \text{for all } x \in \mathbb T\},
	\end{equation*}
	for some scalars $a,b \in \mathbb {R}$, with $a<b$, which will be determined later on.
	\\
	 Given $\phi \in S_{a,b}$, we now consider the problem of constructing $\psi $ as the unique solution of the (linearized) elliptic problem 
	\begin{equation}\label{iteration_system:1}
			 -\partial_x^2  \psi + \tfrac{\mathrm{e}^{\phi  }-1}{\phi }\, \psi = f-1. 
	\end{equation}
	We claim that for any $a<b$ and $ \phi \in S_{a,b}$, there is a unique solution $\psi $ of \eqref{iteration_system:1} which enjoys the bound 
	\begin{equation*}
		\norm {\phi}_{H^2 (\mathbb T)} \lesssim_{a,b} \norm {f-1}_{L^2 (\mathbb T)}.
	\end{equation*}
	To achieve this, we first introduce the bilinear mapping 
	\begin{equation*}
	  	 \begin{aligned}
	  	 	\mathcal L_\phi  
	  	 	:   H^1(\mathbb{T}) \times H^1(\mathbb{T})  &\to \mathbb R
	  	 	\\
	  	 	 (u,v) &\mapsto \int_{\mathbb T} \left( \partial_x  u  \partial_x  v + \tfrac{\mathrm{e}^{\phi  }-1}{\phi }\, uv \right) \ud x.
	  	 \end{aligned} 
	  \end{equation*}
	It is then readily seen that $\mathcal L _\psi$ is well defined and coercive with
	\begin{equation*}
		\mathcal L_\phi (  u, u ) \geqslant C_{a} \norm {u}_{H^1}^2, \quad \text{for all } u\in H^1(\mathbb T),
	\end{equation*}
	where we set  
	$$C_a \bydef  \min \left\{1, \tfrac{\mathrm{e}^{a}-1}{a} \right\} .$$
	Therefore, Lax--Milgram Theorem ensures the existence of a unique solution $\psi$ of \eqref{iteration_system:1}, which belongs to $H^1(\mathbb T)$   and satisfies  the estimate
	\begin{equation}\label{psi bound}
		\norm {\psi}_{C (\mathbb T)} \lesssim \norm {\psi}_{H^1 (\mathbb T)} \leqslant \tfrac{1}{C_a} \norm {f-1}_{L^2 (\mathbb T)},
	\end{equation}
	where we have used in the first inequality  the one-dimensional embedding $H^1 \hookrightarrow C(\mathbb T)$.   Rearranging the terms in \eqref{iteration_system:1}, we obtain
	\begin{equation*}
		-\partial_x^2  \psi  = f-1 -   \tfrac{\mathrm{e}^{\phi  }-1}{\phi }\, \psi.
	\end{equation*}  
	We note  that the right-hand side belongs to $C (\mathbb T)$, and therefore $\partial_x^2 \psi $ belongs to $C(\mathbb T)$, as well, with the desired control in terms of $f$ and $\psi$. This completes the proof of the existence and uniqueness of the solution to \eqref{iteration_system:1}.  In particular, this allows  to define the  nonlinear operator 
	\begin{equation*}
	\begin{aligned}
	  	 	T &: S_{a,b} \to C^2 (\mathbb T)	  	 	\\
	  	 	& \quad  \phi  \mapsto T\phi \bydef \psi .
	  	 \end{aligned} 
	\end{equation*}
{\em\underline{ Step $3 :$} Schauder's fixed point.}
	Our next claim asserts that there are $a,b\in \mathbb R$ for which the operator  $T$ has a fixed point in $S_{a,b}$. 
	This will be achieved via Schauder's Fixed Point Theorem, which yields in turn  that any fixed point is a solution to the original elliptic problem \eqref{elliptic_problem}. To apply this theorem, we need to show that $T$ maps $S_{a,b}$ into itself and is continuous and compact. 
We recall that the remaining assumptions on the set $S_{a,b}$,namely, that it is closed, bounded, and convex, are clearly satisfied by construction.
	\\Let us assume, for the moment, that $T(S_{a,b}) \subset S_{a,b}$ and that  $T$ is continuous from  $S_{a,b}$ to $H^1(\T) \cap S_{a,b}$. Then, due to the compactness of the embedding  $ H^1(\mathbb T) \cap S_{a,b}\hookrightarrow S_{a,b}$, we deduce that $T$ is a compact operator. Therefore, Schauder's Fixed Point Theorem ensures the existence of a fixed point $T\phi = \phi \in C^2 (\mathbb {T})$, which eventually solves \eqref{elliptic_problem}. 
\\
It  remains to show the continuity of $T$ and to find $a$ and $b$ such that $T(S_{a,b}) \subset S_{a,b}$. 
Let $k$ be the smooth positive function defined over $\R$ by 
\begin{equation}\label{k def}
		k(r) \bydef 
	\begin{cases}
	\tfrac{\mathrm \mathrm{e}^ r -1}{r},  &	r \neq 0,\\
	1, & r=0.
	\end{cases}
\end{equation}
For a fixed $f$, assume that $(\phi_1,\psi_1),(\phi_2,\psi_2)$ solve \eqref{iteration_system:1}. Then, we deduce that  
	\begin{equation*}
			 -\partial_x^2  (\psi_1-\psi_2) + k(\phi_1) (\psi_1-\psi_2) +  (k(\phi_1) -k(\phi_2)) \psi_2 = 0. 
	\end{equation*}
Multiplying by 	$\psi_1-\psi_2$ and using \eqref{psi bound} for $\psi_2$ with the fact that $r \mapsto k(r)$ is increasing, we obtain that 
\begin{equation*}
	\|\psi_1-\psi_2\|_{H^1(\T)}^2 \lesssim_{a,f} \|\psi_1-\psi_2\|_{L^2(\T)} \|k(\phi_1) -k(\phi_2)\|_{L^2(\T)}.
\end{equation*}
Consequently, we deduce that
\begin{equation*}
	\|\psi_1-\psi_2\|_{H^1(\T)} \lesssim_{a,f}  \|k(\phi_1) -k(\phi_2)\|_{C(\T)}.
\end{equation*}
The continuity of $T$ follows  from the continuity of the map $r \mapsto k(r)$.
	\\
Let us now seek for suitable values of $a$ and $b$ such that
$$T(S_{a,b}) \subset S_{a,b}.$$
 To this end, we    introduce  
	\begin{equation*}
		c_*\bydef \min_{x\in \mathbb T} \psi (x) = \psi (x_*) \quad \text{and} \quad c^*\bydef \max_{x\in \mathbb T} \psi (x) = \psi (x^*) ,
	\end{equation*}
	for some $x_* $ and $ x^*\in \mathbb T$. 
	Evaluating \eqref{iteration_system:1} at $x_*$ yields 
	\begin{equation*}
		 \begin{aligned}
		 	 \tfrac{\mathrm{e}^{\phi (x_*) }-1}{\phi (x_*)}\, c _*
		 	&= f(x_*)-1 + \partial_x^2  \psi(x_*)
		 	\\
		 	& \geqslant \delta-1  .
		 \end{aligned}
	\end{equation*}
	Hence, using that $r\mapsto k(r)$ is positive and increasing, we obtain that 
	\begin{equation*}
		c_*\geqslant {\tfrac{(1-\delta) a}{1-\mathrm{e}^a}}.
	\end{equation*}
	Thus, choosing 
	$$a \bydef  \log (\delta)$$
	yields that 
	\begin{equation*}
	\psi(x)\geqslant 	c_*\geqslant a, \quad \text{for all } x\in \mathbb T.
	\end{equation*}
	On the other hand, evaluating \eqref{iteration_system:1} at $x^*$ implies that 
	\begin{equation*}
		 \begin{aligned}
		 	 \tfrac{\mathrm{e}^{\phi (x^*) }-1}{\phi (x^*)}\, c ^*
		 	&= f(x^*)-1 + \partial_x^2  \psi(x^*)
		 	\\
		 	& \leqslant  \norm {f-1}_{C (\mathbb T)}  .
		 \end{aligned}
	\end{equation*} 
	Therefore, we obtain that 
	\begin{equation*}
		\psi(x)\leqslant c^* \leqslant \tfrac{a}{\mathrm{e}^{a}-1}\norm {f-1}_{C(\mathbb T)} = \tfrac{\log(\frac{1}{\delta})}{1-\delta} \norm {f-1}_{C (\mathbb T)} \bydef b ,
	\end{equation*}
	for all $x\in \mathbb T$. This establishes that $T$ maps $S_{a,b} $ into itself with  the above choice of parameters $a$ and $b$. 
\\	
{\em\underline{ Step $4 :$}  Estimates.} Evaluating \eqref{elliptic_problem} at the maximum and the minimum of $\phi$, we obtain the pointwise estimate
\begin{equation*}
	\log \left (\min_{y\in \mathbb T} f(y)\right)  \leqslant   \phi(x)\leqslant \log \left (\max_{y\in \mathbb T} f(y)\right),
\end{equation*} 
for any $x \in \T$.
Let us now write \eqref{elliptic_problem} on the form 
\begin{equation*}
 -\partial_x^2 \phi + k(\phi) \phi =  f-1.
\end{equation*}
Multiplying by $\phi$ and using that $k(\phi) \geqslant k(\log(\delta))$, we obtain that  
\begin{equation*} 
\norm {\phi}_{C(\mathbb T)} \lesssim	\norm {\phi}_{H^1(\mathbb T)} \lesssim_\delta \norm {f-1}_{L^2(\mathbb T)}.
\end{equation*}
The $H^2(\T)$ and $C^2(\T)$ estimates follow from the equation, thereby completing the proof of the proposition. 
\end{proof}

\begin{remark}
	In the preceding proposition, if we assume in addition that  $f\in H^s(\mathbb T)$, then we can obtain an estimate for $\norm {\phi}_{H^{s+2}(\mathbb{T})}$ using  energy-type estimates. 
\end{remark}

\subsection{On a companion linearized elliptic operator}\label{elliptic_problem_lin}

  It will be apparent through our analysis later on  that the nonlinear nature of the operator $\mathcal H$ and its inverse is the main source of difficulty to establish certain important properties of solutions to \eqref{main_EQ} (see Lemmas \ref{lamma:sign}, \ref{lem:strict_monotonicity} and \ref{lem:amplitude} below).
  This is essentially because of the absence of an explicit formula for $\mathcal{H}^{-1}$.  However, we will see that a fine understanding of a specific  linearized operator will allow us to derive many key properties of $\mathcal H$ and it inverse. More precisely,  our first aim now is to analyze the inverse of linear operator 
$$\mathcal L_\lambda  \bydef \lambda \mathrm{Id} -\partial_x^2,$$
for a given constant $\lambda >0$. It is well known that, if $h \in C(\T)$, then the equation 

\begin{equation*}
	\mathcal L_\lambda\phi=h
\end{equation*}
has a unique solution $\phi$ given by 
\begin{equation*}
	\phi \bydef \mathcal L_\lambda^{-1} h \bydef G_{\lambda} \ast h,
\end{equation*}
where 
\begin{equation}\label{Gdef}
	G_{\lambda}(x) \bydef \frac{\cosh \left(\sqrt{\lambda} \left( x-L \lfloor \tfrac{x}{L} \rfloor - L/2 \right) \right) }{2 \sqrt{\lambda} \sinh \left(\sqrt{\lambda}   L/2\right)},
\end{equation}
where we recall that $L$ is the period of the torus $\mathbb T$, and $ \lfloor \cdot  \rfloor$ denotes the floor function, i.e.  $ \lfloor x  \rfloor = \max \{k \in \Z, k \leqslant x\}$. 
It is a classical fact  that the function $G_\lambda \in W^{1,\infty}(\T)$ is positive, $L$-periodic, even and increasing on $(-\nicefrac{L}{2},0)$.
Moreover, the function $\mathcal L_\lambda  ^{-1} h $ has the same parity as $ h $. 
Note that the operator $\mathcal L_\lambda  ^{-1}$ can be recast, for odd functions $\tilde h$, as  
\begin{equation*}
	(\mathcal L_\lambda  ^{-1} \tilde h)(x) = \int_{- \nicefrac{L}{2}}^0 \left( G_{\lambda}(x-y)- G_{\lambda}(x+y) \right) \tilde{h} (y)\, \ud y,
\end{equation*} 
for any $x\in \mathbb T$. Let us begin with a basic discussion on some properties of the kernel $G_\lambda$. 

\begin{lemma}\label{lem:G}
Let $\lambda \in (0,\infty)$ and $x \in (-\nicefrac {L}{2},0)$, then 
\begin{align*}
G_{\lambda}(x-y)- G_{\lambda}(x+y)>0, \quad \text{for any } y \in (-L/2,0), \\
G_{\lambda}(x-y)- G_{\lambda}(x+y)=0, \quad \text{for  } y \in \{-L/2,0\}.
\end{align*}
\end{lemma}

\begin{proof}	
Due to the parity of $G_\lambda$, we can assume that $y<x$. 
We can consider  two cases. On the one hand, if $x+y \in  [-\nicefrac{L}{2},0)$, then, using that the kernel  $G_{\lambda}$ is even and decreasing on $(0, \nicefrac {L}{2})$, combined with the inequality $0<x-y<|x+y|<\nicefrac {L}{2}$, we  obtain  that  
$$G_{\lambda}(x-y)- G_{\lambda}(x+y)>0.$$
On the other hand, in the second case when $x+y \in  (-L,-\nicefrac {L}{2}]$, we use in addition the $L$-periodicity of the kernel $G_{\lambda}$ and the inequality $ 0<x-y < L+x+y< \nicefrac {L}{2}$ to end up with the same preceding inequality. This concludes the proof of the lemma.
\end{proof}

We show now that the operator $\mathcal L_\lambda^{-1}$ transforms even, non-trivial and non-decreasing functions on the half period to increasing functions on the same interval. We also present local convexity properties.

\begin{lemma}\label{parity_elliptic}
	Let $\lambda \in (0, \infty)$ be fixed, and  consider an $L$-periodic function $ h \in C(\mathbb T)$.  Moreover, if $ h  $ is non-constant, even and non-decreasing on $(-\nicefrac{L}{2} ,0)$, then it holds that 
\begin{equation*}
	\partial_x\mathcal L_\lambda  ^{-1} h  >0 ,
\end{equation*}
on $(-\nicefrac{L}{2},0)$, and   
\begin{equation*}
	\partial_x^2 \mathcal L_\lambda  ^{-1}h  (0) <0 \qquad \text{and} \qquad\partial_x^2 \mathcal L_\lambda  ^{-1}h  (-L/2) >0.
\end{equation*}
\end{lemma}

\begin{proof}
Let  $x\in (-L/2,0)$, then straightforward computations, using in particular that $ G_{\lambda}'$ is odd, yield 
\begin{align*}
	\partial_x\mathcal L_\lambda  ^{-1} h(x)
	&= \int_{-L/2}^{L/2} G_{\lambda}'(y) h(x-y)\, \ud y\\
	&= \int_{-L/2}^{0} G_{\lambda}'(y) (h(x-y)-h(x+y))\, \ud y.
\end{align*}
Therefore, as $h$ is non-constant, following the proof of Lemma \ref{lem:G}, we can show that 
\begin{align*}
	h(x-y)-h(x+y) &\geqslant 0 \qquad \text{for any } y \in (-L/2,0), \\
	h(x-y)-h(x+y) &>0 \qquad \text{for some } y \in (-L/2,0).
\end{align*}
Thus, since $G_{\lambda}$ is increasing on $(-\nicefrac{L}{2},0)$, it follows that  
$$\forall x\in (-\nicefrac{L}{2},0),\quad \partial_x\mathcal L_\lambda  ^{-1} h (x) >0.
$$
As to the sign of $ \partial_x^2 \mathcal L_\lambda  ^{-1}h  (0)$, we use the fact that  
$$\int_{-L/2}^{L/2} G_{\lambda}(x)\, \ud x = \frac {1}{\lambda} $$
 to obtain that 
\begin{align*}
	\partial_x^2 \mathcal L_\lambda  ^{-1}h  (0) &=   \lambda \mathcal L_\lambda  ^{-1}h  (0) - h(0)\\
	&= \lambda  \left(\int_{-L/2}^{L/2} G_{\lambda}(y) h(y)\, \ud y - \int_{-L/2}^{L/2} h(0) G_{\lambda}(y)\, \ud y  \right) <0.
\end{align*}
Similarly, we compute that 
\begin{align*}
	\partial_x^2 \mathcal L_\lambda  ^{-1}h  (-L/2) &=   \lambda \mathcal L_\lambda  ^{-1}h  (-L/2) - h(-L/2)\\
	&= \lambda  \left(\int_{-L/2}^{L/2} G_{\lambda}(y) h(y+L/2)\, \ud y - \int_{-L/2}^{L/2} h(-L/2) G_{\lambda}(y)\, \ud y  \right) >0,
\end{align*}
thereby establishing the claimed sign of $\partial_x^2 \mathcal L_\lambda  ^{-1}h  (-L/2) $. This  concludes the proof.
\end{proof}

\section{Properties of traveling waves}\label{section:preleminary:traveling_waves}

In this section, we develop a self-contained analysis of several fundamental properties of any a priori solution 
  to the nonlinear equation \eqref{main_EQ}. At this stage, the existence of solutions to \eqref{main_EQ} has not yet been established. Nevertheless, the results derived in this section will play a crucial role in the subsequent construction of a ``singular endpoint'' solution to \eqref{main_EQ}. 
  
  We recall that $p$  satisfies  the conditions stated in Section~\ref{section:pressure}, and that the function $\mathcal G_c$ introduced in \eqref{G_c:def} is decreasing on $(0,a^*(c))$ where $a^*(c)$ denotes its critical point. 

Finally, for the mere sake of convenience, and to avoid any possible confusion in notations in the proofs below, we point out that we choose to consider the change the notation $(c,\rho)\to (c,f)$ from now on when we refer to the solution of \eqref{main_EQ}.

\subsection{Preliminary analysis of the traveling waves}

We begin with a preliminary discussion of the properties satisfied by any solution   to the equation \eqref{main_EQ}. For  clarity, the results are organized into a series of propositions and lemmas, which will be instrumental in the subsequent analysis.

\begin{proposition}[Smoothness away from the critical point]\label{prop:smoothness:1}
Let $ (c,f) $ be a solution of the equation \eqref{main_EQ}, with  $f$ being a continuous function satisfying the bounds 
\begin{equation*}
	0<\delta \leqslant f(x) \leqslant a^*(c), \quad \text{for any } x \in \T.
\end{equation*}	
Then, $f$ is $C^\infty$-smooth   in the open set 
\begin{equation*}
I \bydef f^{-1}\left (\delta, a^*(c) \right) = \{x \in \T, \delta <f(x) <a^*(c) \}.
\end{equation*} 
\end{proposition}

\begin{proof}
We recall that $\phi=\mathcal{H}^{-1}(f)$ satisfies the equation 
\begin{equation*}
	\mathrm{e}^\phi-\partial_x^2 \phi =f.
\end{equation*}
Since $f\in C(\T)$, then Proposition \ref{prop:elliptic} implies that $\phi \in C^2(\T)$.
Let $\mathcal G_c $ be the function defined in \eqref{G_c:def}, which is locally analytic on $(0,\infty)$. Note that the map 
\begin{equation*}
	\mathcal G_c : \left (\delta, a^*(c) \right) \to (\mathcal G_c(a^*(c)), \mathcal G_c(\delta))
\end{equation*}
 is invertible and its inverse is of class $C^\infty$ (it is in fact locally analytic as well). Therefore,  one can recast \eqref{main_EQ} as
	\begin{equation}\label{f=G^-1}
		f= \mathcal G_c^{-1}\left( - \phi\right).
	\end{equation}
On the domain $I$, $f$ takes its values in the interval $(\delta,a^*(c))$ and therefore $-\phi$, which coincides with $\mathcal{G}_c(f)$, takes values in $(\mathcal G_c(a^*(c)), \mathcal G_c(\delta))$. Thus, equation \eqref{f=G^-1} gives that $f$ is of class $C^2$ on $I$, which  implies in turn  that $\phi$ is of class $C^4$ on $I$ as well. The desired result follows using a bootstrap argument.  
\end{proof}

In the next result, we establish that, for a large value of $c$, the equation \eqref{main_EQ} does not admit non-trivial solutions.  This will come in handy later in a subsequent section.

\begin{proposition}[Return to equilibrium for large speeds]\label{prop:bound:c}
Let  $\delta \in (0,1)$ be fixed. There is a constant $K_{\delta}>0$, depending only on $\delta$, such that,
	if $(c,f)$ solves \eqref{main_EQ}, where $f$ is continuous such that 
	\begin{equation*}
		 \delta \leqslant f(x) \leqslant a^*(c), \quad \text{for all } x\in \mathbb T,
	\end{equation*}
	 and 
	\begin{equation*}
		|c|> K_{\delta},
	\end{equation*}
	then  $	f\equiv 1. $
\end{proposition}

\begin{proof}
We denote the maximum of $f$ by $M$, i.e., 
\begin{equation*}
	M \bydef \max_{x\in \T} f(x)\leqslant a^*(c).
\end{equation*}
	Taking the $L^2$-inner product of \eqref{main_EQ} with $f-1$ yields that 
	\begin{equation*}
		\begin{aligned}
			\frac{c^2}{2} \int_{\mathbb T} \frac{f(x)+1}{f^2(x)}  |f(x)-1|^2\, \ud x = \int_{\mathbb T} (p(f) (x)-p(1)) (f(x)-1)\, \ud x + \int_{\mathbb T} \mathcal H^{-1}(f)(x) (f (x)-1)\, \ud x.
		\end{aligned}
	\end{equation*}
	Combining   the elementary bounds
		\begin{equation*}
		\begin{aligned}
			|p(f) (x)-p(1) |
			&\leqslant \max_{\xi \in [\delta , M]} p'(\xi)  |f (x)-1| 			
		\end{aligned}
	\end{equation*}
	and
	\begin{equation*}
		\frac{c^2}{2} \int_{\mathbb T} \frac{f(x)+1}{f^2(x)}  |f(x)-1|^2\, \ud  x \geqslant \frac{c^2}{2 M}    \norm {f-1}_{L^2(\mathbb T)}^2,
	\end{equation*}
	and employing   the control of $\mathcal H^{-1} (f)$ from Proposition \ref{prop:elliptic} 
	\begin{equation*}
		\norm {\mathcal H^{-1}(f)}_{L^2(\mathbb T)} \lesssim_\delta \norm {f-1}_{L^2(\mathbb T)},
	\end{equation*}
	 we find that 
	\begin{equation*}
		 \begin{aligned}
		 	\frac{c^2}{M} \norm {f-1}_{L^2(\mathbb T)}^2 
		 	&\lesssim_\delta  \left( \max_{\xi \in [\delta , M]} p'(\xi)+1 \right)   \norm {f-1}_{L^2(\mathbb T)}^2.
		 \end{aligned}
	\end{equation*}
Now, we define  the increasing (and bijective) map $J : (0,\infty) \to (0,\infty)$ as 
\begin{equation*}
	J(z) \bydef 
	\displaystyle\begin{cases}
		z \left( \displaystyle\max_{\xi \in [\delta , z]} p'(\xi)+1 \right), & z \geqslant \delta,\\ ~~\\
		 z(p'(\delta)+1), & 0<z \leqslant \delta,
	\end{cases}
\end{equation*}
and we notice that 
	\begin{equation}\label{contraction}
		 	\frac{c^2}{J(M)} \norm {f-1}_{L^2(\mathbb T)}^2 
		 	\leqslant B_{\delta}    \norm {f-1}_{L^2(\mathbb T)}^2,
	\end{equation}	
for some constant $B_\delta >0$ depending only on $\delta$.	\\
Let $x_0 \in \T$, such that $f(x_0)=M$, evaluating \eqref{main_EQ} at $x_0$ yields that  
\begin{equation*}
	     \frac{c^2}{2} \left (\frac{1}{M^2} -1\right) + p(M)-p(1) + \mathcal H^{-1}(f)(x_0) =0.
\end{equation*}
Therefore, we find, by using Proposition \ref{prop:elliptic},  that  
\begin{align*}
	     \frac{c^2}{M} &= c^2 M - 2 M (p(M)-p(1)) - 2  M \mathcal H^{-1}(f)(x_0),\\
	    & \geqslant  M^4 p'(M) - 2 M (p(M)-p(1)) - 2 M \log(M),
\end{align*}
where we used that $c^2 = a^*(c)^3 p'(a^*(c)) \geqslant M^3p'(M)$, which follows from the definition \eqref{a*def} and the fact that the map $\xi \mapsto \xi^3 p'(\xi)$ is increasing.	
Dividing now by $\max_{\xi \in [\delta , M]} p'(\xi)+1$, it follows that 
\begin{equation}\label{c2 over J}
	\frac{c^2}{J(M)} \geqslant W(M),
\end{equation}
where $W$ is defined in \eqref{Wdef} and the subscript (in terms of $\delta$) in its definition is dropped for   simplicity.
 Let $\tilde{W}$ be a continuous and increasing function on $[\delta ,\infty)$, such that
\begin{equation*}
	 \tilde{W}(\xi) \leqslant W(\xi)  \quad \text{for any } \xi \geqslant \delta   \qquad \text{and} \qquad \lim_{\xi \to \infty} \tilde{W}(\xi)=\infty.
\end{equation*}
One can take $B_\delta$ large enough in \eqref{contraction} such that $B_\delta >\tilde{W}(\delta)$ (note that $B_\delta$ can depend on $p$ as well in this case).
Take now $c$ large enough such that 
\begin{equation}\label{c large}
	\tilde{W} \left( J^{-1} \left( \tfrac{c^2}{B_\delta} \right) \right) > B_\delta. 
\end{equation}
Our claim now  is to show, under the preceding condition on $c$ being sufficiently large,  that  
\begin{equation*}
	\frac{c^2}{J(M)} > B_\delta,
\end{equation*}
which, with \eqref{contraction}, will imply that $f\equiv 1$. 
Looking for a contradiction, we assume that 
\begin{equation*}
	\frac{c^2}{J(M)} \leqslant B_\delta,
\end{equation*}
which, in view of \eqref{c large}, implies that
\begin{equation*}
\delta< \tilde{W}^{-1}(B_\delta) < J^{-1} \left( \tfrac{c^2}{B_\delta} \right)  \leqslant M.	
\end{equation*}
Therefore, using \eqref{c large}, again, together with   \eqref{c2 over J}, we obtain that 
\begin{align*}
	B_\delta & <  \tilde{W} \left( J^{-1} \left( \tfrac{c^2}{B_\delta} \right) \right)  \leqslant \tilde{W}(M) \leqslant W(M) \leqslant \frac{c^2}{J(M)} \leqslant B_\delta,
\end{align*}
thereby leading to a contradiction.  This completes the proof of the proposition.
\end{proof}

\begin{proposition}[Lower bound on $\delta$]\label{prop:delta:min} 
There exists $\delta_0>0$ depending only on $p$ such that, for any solution $(c,f)$   of \eqref{main_EQ} satisfying 
 	\begin{equation*}
 		a^*(c) \geqslant 1, \qquad    \min_{x\in \mathbb T} f(x)>0, \qquad f \in C(\T),
 	\end{equation*}
  it holds that  
 	\begin{eqnarray*}
 		 2 \delta_0 \leqslant \min_{x\in \mathbb T} f(x).
 	\end{eqnarray*}
 \end{proposition}

\begin{proof}
  Towards a contradiction, we assume that there exists a sequence $(c_j,f_j)_{j\in \mathbb N}$ of solutions of \eqref{main_EQ} such that 
   	\begin{equation}\label{mj:to:0}
 		a^*(c_j) \geqslant 1, \quad  \text{for all } j \in \mathbb N,  \quad \text{and}  \quad \lim_{j\to \infty} m_j =0,  
 	\end{equation}
 	where
 	\begin{equation*}
 	    m_j \bydef  \min_{x\in \mathbb T} f_j(x) >0.
 	\end{equation*}
Let $x_j \in \mathbb T$ be such that $f_j(x_j)=m_j$. Thus, it follows by evaluating \eqref{main_EQ} at $x_j$ that
 	\begin{equation*}
 	\frac{c_i^2}{2}(m_j^{-2}-1)+ p(m_j)-p(1) = - \mathcal H^{-1}(f_j)(x_j).
 	\end{equation*}
 	Therefore,   employing the lower bound of  $c_j^2 $, together with the maximum principle from  Proposition \ref{prop:elliptic}, we obtain (for $j\gg1$) that 
  	\begin{equation*}
 		\tfrac{\left((a^*)^{-1}(1)\right)^2}{2}(m_j^{-2}-1)+p(m_j)-p(1) \leqslant \log \tfrac{1}{m_j} .
 	\end{equation*}	
 	Multiplying both sides by $m_j^2$ and applying  the result
	\begin{align}\label{xi^2p}
	\lim_{\xi \to 0^+} \xi^2 p(\xi) =0,
	\end{align}
 we obtain, as $j \to \infty$, that
 	\begin{equation*} 
 		\tfrac{\left((a^*)^{-1}(1)\right)^2}{2} \leqslant 0.
 	\end{equation*}
This would imply, in view of \eqref{a*def}, that $p'(1)=0$, which contradicts the first assumption on $p$ from  \eqref{p_conditions1}. This  shows that the scenario \eqref{mj:to:0} cannot occur, thereby concluding  the proof of the proposition.
It remains to verify \eqref{xi^2p}. First, we recall from \eqref{p_conditions1} that 
\begin{equation*}
		\lim_{\xi \to 0^+} \xi^3 p'(\xi) =0 \qquad \text{and} \qquad p'(\xi)>0	.
\end{equation*}
Let $\varepsilon>0$, there exists $\eta>0$ such that, for any $\xi \in (0,\eta)$, we have $p'(\xi) \leqslant \varepsilon/\xi^3$. Therefore, it holds that 
\begin{equation*}
	p(\eta)- p(\xi)  = \int^\eta_\xi p'(\sigma)\, \ud \sigma \leqslant  \tfrac{\varepsilon}{2\xi^2}- \tfrac{\varepsilon}{2\eta^2},
\end{equation*}
for any $\xi \in (0,\eta)$, which implies that 
\begin{equation*}
	\limsup_{\xi \to 0^+} \xi^2 p(\xi) \geqslant 0.
\end{equation*}
On the other hand, we use the monotonicity of $p$ to find that  
\begin{equation*}
	\liminf_{\xi \to 0^+} \xi^2 p(\xi)  \leqslant \liminf_{\xi \to 0^+} \xi^2 p(1) =0.
\end{equation*}
Combining the last two limits ends the proof of \eqref{xi^2p}.
\end{proof}

In the next proposition, we establish  uniform bounds for the solution $f$ in $W^{\frac12,\infty}(\T)$, and derive a sharp asymptotic expansion for any even solution $f$ to \eqref{main_EQ} in a neighborhood of its critical point. This, in particular, provides a precise description of the slope at the most singular point of $f$.
 
\begin{proposition}\label{prop:uniform_bound and Holder regularity}
	 Let $(c,f)$ be a solution to \eqref{main_EQ} satisfying, for some $\delta\in(0,1)$, that
	\begin{equation*}
		 \delta \leqslant   f(x) \leqslant a^*(c), \quad \text{for all } x\in \mathbb T.
	\end{equation*}
	Then, there exists a constant $C=C(\delta)>0$ (that is independent of $c$), such that 
	\begin{equation*}
		\|f\|_{L^\infty(\T)} + \sup_{x \neq y}  \frac{|f(x)-f(y)|}{|x-y|^\frac{1}{2}}\leqslant C.
	\end{equation*}
	Moreover, if $f$ is even and $f(0)=a^*(c)$, then it holds that 
\begin{equation*}
	a^*(c)-f(x)= \theta |x|+ \mathcal O (|x|^{3/2}),
\end{equation*}
for $|x|\ll 1$, where 
\begin{equation*}
	\theta \bydef \sqrt{\tfrac{-\phi''(0)}{\mathcal G_c''(a^*(c))}}  =  \sqrt{\tfrac{a^*(c)-\mathrm{e}^{-\mathcal G_c(a^*(c))}}{\mathcal G_c''(a^*(c))}}.
\end{equation*}
\end{proposition}

\begin{remark}
	Note that, by virtue of   Proposition \ref{prop:delta:min}, if we further assume that $a^*(c) \geqslant 1$, then the hypothesis $f(x) \geqslant \delta>0$ can be dropped.
\end{remark}

\begin{remark}
The uniform $\frac12$-H\"older bound will be upgraded to the Lipschitz regularity for solutions with half-period monotonicity in Proposition~\ref{pro:Lipregularity}, below. Accordingly, the error term in the asymptotic expansion near the critical point from the preceding proposition  will also be  refined in Proposition~\ref{pro:Lipregularity}, later on .
\end{remark}

\begin{proof}	

We proceed in three separate steps.

\paragraph{\underline{\em $L^\infty$ bound}}

Let   $K_{\delta}$ be the constant from Proposition \ref{prop:bound:c}, and we consider two cases. If $|c| \leqslant K_{\delta}$, then $f(x) \leqslant a^*(K_{\delta})$. Otherwise, if $|c| > K_{\delta}$, then $f \equiv 1$, by Proposition \ref{prop:bound:c}. In either case, $f$ is bounded above by a constant that does not depend on the speed $c$. This concludes the proof of the $L^\infty$ bound of $f$.

\paragraph{\underline{\em H\"older regularity}}

Recall that $\phi = \mathcal H^{-1}(f)$, and by virtue of Proposition \ref{prop:elliptic}, one can easily deduce the bound 
\begin{equation}\label{bound:phi:Lip*}
	|\phi(x)-\phi(y)| \lesssim_{\delta} |x-y|,
\end{equation} 
which will be used later on at the end of the proof.

Now, evaluating \eqref{main_EQ} at two different points $x,y\in \mathbb T$ and taking the difference of the resulting equations   yields that  
	\begin{equation}\label{equa:difference}
		\mathcal G_c \big ( f(x)\big ) - \mathcal G_c \big ( f(y)\big ) = \mathcal H^{-1}(f )(y)-\mathcal H^{-1}(f )(x)=\phi(y)-\phi(x),
	\end{equation}
	where we recall that the function $\mathcal G_c$ is given by
	\begin{equation*}
		\mathcal G_c (\xi)= \tfrac{c^2}{2}(\xi^{-2}-1) + p(\xi)-p(1), \quad \xi>0.
	\end{equation*}
	One can check by a direct computation that $\mathcal G_c$ fulfills the conditions stated in Lemma \ref{lemma:MVT}, ensuring in particular the lower bound 
	\begin{equation*}
		|\mathcal G_c(f(x))- \mathcal G_c(f(y))| \geqslant  \tfrac{1}{2}\min _{z\in \left[\delta,\|f\|_{C(\T)}\right]} | \mathcal G_c''(z)|  |f(x)-f(y)|^2,\quad \text{for all } x,y \in \T.
	\end{equation*}
We find, on the one hand, that
\begin{align}\nonumber
	\inf_{\xi \in [\delta, a^*(c)]}  \mathcal G_c''(\xi)
	&= \inf_{\xi \in [\delta, a^*(c)]} \left( 3 c^2 \xi^{-4}+p''(\xi) \right)\\ \nonumber
	 &= \inf_{\xi \in [\delta, a^*(c)]} \left( 3 a^*(c)^3 p'(a^*(c))\xi^{-4}+ p''(\xi) \right)\\ \nonumber
	 &\geqslant \inf_{\xi \in [\delta, a^*(c)]}  \left( 3 \xi^{-1} p'(\xi)+ p''(\xi) \right)\\ \label{infG''}
	 &\geqslant \inf_{\xi \in [\delta, a^*(K_{\delta})]}  \left( 3 \xi^{-1} p'(\xi)+ p''(\xi) \right)>0,\nonumber
\end{align}	
where we used \eqref{a*def} together with the  condition \eqref{p_conditions1}. Therefore, we deduce, in view of \eqref{equa:difference}, that 
\begin{equation*}
		|f(x)-f(y)|^2 \lesssim_{\delta}|\phi(x)- \phi(y)| ,\quad \text{for all } x,y \in \T.
	\end{equation*}
By employing the bound \eqref{bound:phi:Lip*}, we obtain the H\"older control of $f$, thereby concluding the  proof of  the desired estimate.

\paragraph{\underline{\em Asymptotic behavior}}

In order to prove the asymptotic behavior near the absolute maximum, we start by noting that the condition  $f(0)=a^*(c)$ forces $c$ to be bounded. Indeed, Proposition \ref{prop:bound:c} shows that if $c$ is large enough (more precisely, larger than the constant $K_{\delta}$ from that proposition), then this would imply that $f(0)=a^*(c)=1$, which contradicts the assumption  \eqref{a^*lim}. We may then assume that 
\begin{equation*}
	|c| \leqslant K_{\delta},
\end{equation*}
where the constant $K_{\delta}$ is the one from Proposition \ref{prop:bound:c}.
Evaluating  \eqref{main_EQ} at $x=0$ and at an arbitrary $x\in \T$, with $|x|\ll1$, and taking the difference between the resulting equations   yields that  
	\begin{equation*}
		\mathcal G_c \big ( a^*(c) \big ) - \mathcal G_c \big ( f(x)\big ) =\phi(x)-\phi(0).
	\end{equation*}		
On the one hand, by the H\"older regularity of $f$  established before, we deduce that $\phi \in C^2$ and, eventually, that  $\phi'' \in C^\frac{1}{2}(\T)$. Therefore, the expansion
	\begin{equation*}
	\phi(x)-\phi(0)=\tfrac12 \phi''(0)x^2+ \mathcal{O} (|x|^{5/2}) 
	\end{equation*}	
follows, where we used  that $\phi'(0)=0$, which is a consequence of parity of the function $\phi$. 
\\On the other hand, we write, for some $\xi \in [f(x),a^*(c)]$, that 
\begin{equation*}
 \mathcal G_c \big ( f(x)\big )-\mathcal G_c \big ( a^*(c) \big ) = \tfrac12  \mathcal G_c'' (\xi) \big(a^*(c)-f(x)\big)^2,
\end{equation*}
which follows from  Taylor expansion combined with  the fact that $\mathcal G_c' (a^*(c))=0 $. 
All in all, we deduce from the preceding identities  that 
\begin{equation*}
{\big(a^*(c)-f(x)\big)^2= \frac{1}{ \mathcal G_c'' (\xi)}  \left(  -\phi''(0)x^2+ \mathcal{O} (|x|^{5/2}) \right)}.
\end{equation*}
Observe now that the H\"older regularity of $f$, together with the Mean Value Theorem, allows us to obtain the bound
\begin{equation*}
\big| \mathcal G_c'' (\xi) - \mathcal G_c'' (a^*(c)) \big|  \leqslant
\sup_{\eta \in [\delta, a^*(K_{\delta})]} \left| \mathcal G_c''' (\eta) \right| \left( a^*(c)-f(x) \right) \lesssim_{\delta} {|x|^{1/2}}
\end{equation*}
which allows us to find  that 
\begin{equation*}
	\frac{1}{\mathcal G_c'' (\xi)}=\frac{1}{\mathcal G_c'' (a^*(c))+\mathcal O (|x|^{1/2})} = \frac{1}{\mathcal G_c'' (a^*(c))} +\mathcal O (|x|^{1/2}).
\end{equation*}
Therefore, we conclude that  
\begin{equation*}
 \big(a^*(c)-f(x)\big)^2	=\tfrac{-\phi''(0)}{ \mathcal G_c'' (a^*(c))}x^2+ \mathcal{O} (|x|^{5/2}),
\end{equation*}
which leads to the expansion  
\begin{equation*}
a^*(c)-f(x)	=\sqrt{ \tfrac{-\phi''(0)}{ \mathcal G_c'' (a^*(c)) }} |x|+ \mathcal{O} (|x|^{3/2}).
\end{equation*}
Finally, evaluating both equations $\mathcal H(\phi)=f$ and \eqref{main_EQ} at $x=0$ gives that
\begin{equation*}
	-\phi''(0)=f(0)-\mathrm{e}^{\phi(0)}=f(0)-\mathrm{e}^{-\mathcal G_c(f(0))},
\end{equation*}
which allows us to compute the coefficient of $|x|$ in the last asymptotic expansion.  This completes  the proof of the proposition.
\end{proof}

It will be important in the sequel to understand how any solution $f$ to \eqref{main_EQ} relates to the equilibrium configuration $f \equiv 1$. The following lemma provides a first insight into this relationship. In particular, it shows that any nontrivial solution must oscillate around the equilibrium level. This observation will play a useful role in the global analysis carried out in Section \ref{sec:global}, for instance in the study of compactness properties along the bifurcation curve (see Proposition \ref{prop:openness:cone}).

\begin{lemma}[Oscillation around equilibrium]\label{lamma:sign}
	Let $(c,f)$ be a solution of \eqref{main_EQ}, where $f$ is continuous and satisfies 
	\begin{equation*}
		f(x_0)=1, \quad \text{for some } x_0\in \mathbb T.
	\end{equation*}
	Then, either
	\begin{equation*}
		f(x)=1, \quad \text{for all } x\in \mathbb T,
	\end{equation*}
	or the function 
	\begin{equation*}
		x\mapsto f(x)-1
	\end{equation*}
	changes its sign on $\mathbb T$.
\end{lemma}

\begin{proof}
We recall the notation $ \phi=\mathcal H^{-1}(f)$ where $\phi $ is governed by the equation 
 \begin{equation}\label{phi:equa}
	 	-\partial_x^2 \phi + \mathrm{e}^\phi = f.
	 \end{equation}
Further introduce, for any $\lambda >0$, the continuous function  $h$ by setting
\begin{equation}\label{h:equa}
	  h \bydef \lambda \phi -\partial_x^2 \phi.
\end{equation}
	Assuming that $f(x_0)=1$, for some $x_0 \in \mathbb T$, yields, from \eqref{main_EQ}, that 
	\begin{equation*}
		\phi(x_0)=0,
	\end{equation*}
	for that particular $x_0$. Thus, since
	$$\phi = G_{\lambda} \ast h,$$
	 where $G_{\lambda}$ is a positive kernel defined in \eqref{Gdef}, we have that either $h \equiv 0$ or $h$ is sign changing.

	 Now, if $h \equiv 0$ for some $\lambda >0$, then $\phi \equiv 0$ and therefore $f \equiv 1$, which follows from  \eqref{phi:equa}.
	 It then remains to show that if $h$ changes its sign for all $\lambda>0$, then $f-1$ does so, as well. To that end, let us first take  $\lambda=1$, and pick $x_+ \in \mathbb T$ such that 
	\begin{equation*}
		h(x_+)>0.
	\end{equation*}
	 Thus, it follows by combining \eqref{phi:equa} and \eqref{h:equa}  that 
		\begin{equation*}
		f(x_+)=h(x_+)+ \mathrm{e}^{\phi(x_+)} -\phi(x_+)>1,
	\end{equation*}
	where we used the elementary bound 
	\begin{equation*}
		\mathrm{e}^\xi-\xi \geqslant 1, \quad \text{for all } \xi\in \mathbb R.
	\end{equation*}
	We shall now prove that $\inf_{x\in \mathbb T} f(x)<1$, and we begin with  considering  the case 	$\inf_{x\in \mathbb T} \phi < 0$ at first. To proceed further, we employ Proposition \ref{prop:elliptic} to write that 
	\begin{equation*}
			 \log (\inf_{x\in \mathbb T} f(x)) \leqslant \inf_{x\in \mathbb T}\phi(x)<0,
	\end{equation*}
	which gives that $\inf_{x\in \mathbb T} f(x)<1$ in the case when $\inf_{x\in \mathbb T} \phi < 0 $. On the other hand, in the event when $\inf_{x\in \mathbb T} \phi \geqslant 0$,  we make the choice   of $\lambda $ as
	\begin{equation*}
	\lambda \bydef \sup_{x\in \mathbb T} k(\phi(x)),
	\end{equation*}
where $k$ is defined in \eqref{k def}.
Accordingly, by picking any $x_-\in \mathbb T$ such that $h(x_-)<0$, it follows, from \eqref{phi:equa} and \eqref{h:equa} again, that  
$$f(x_-)=h(x_-) + \mathrm{e}^{\phi(x_-)} -\lambda\phi(x_-) <1,$$
where the last inequality on the right-hand side is based on the particular choice of the parameter $\lambda$. This concludes the proof of the lemma. 
 \end{proof}

\subsection{Solutions with half-period monotonicity}
In this section, we focus on a class of solutions $(c,f)$ to \eqref{main_EQ} exhibiting a monotonicity property on half of the period, namely that $f$ is non-decreasing on $(-\nicefrac{L}{2},0)$. This structural assumption turns out to be particularly relevant in the analysis of the equation and will allow us to derive additional qualitative properties of the solutions.
\\
Importantly, the solutions that will be constructed later in the paper by means of bifurcation arguments will be shown to satisfy this half-period monotonicity. As such, the results established in this section will play a key role in the subsequent analysis.

\begin{lemma}\label{lem:strict_monotonicity}
	Let $L >0$ and $(c,f)$ be any given  solution to \eqref{main_EQ}. Assume that $f\in C^1(\R - L\Z) \cap C(\T)$ is $L$-periodic,  even, non-constant, non-decreasing on $(-\nicefrac L2,0)$ and satisfies the upper bound
\begin{equation*}
	f(x)\leqslant a^*(c).
\end{equation*}
Then, it holds that  
	\begin{equation*}
		f'(x)>0\qquad \text{ and } \qquad f(x)< a^*(c),  
	\end{equation*}
	for any $x\in (\nicefrac L2,0)$, while 
	\begin{equation*}
		\partial_x^2 \mathcal{H}^{-1}(f)(0)<0.
	\end{equation*}
Moreover, if $f\in C^2(\T)$, then one has that 
\begin{equation*}
	f''(0)<0 \qquad \text{and} \qquad f''(-L/2)>0.
\end{equation*}
\end{lemma}

\begin{proof}
We recall the notation $\phi = \mathcal{H}^{-1}(f)$, where $\phi$ satisfies \eqref{phi:equa}, as well as the definition of the continuous function $h$ given in \eqref{h:equa}, for any $\lambda>0$.
Now, differentiating \eqref{main_EQ} with respect to $x$ on $(-\nicefrac L2,0)$, we obtain that 
\begin{equation}\label{diff:main}
\mathcal	G_c'(f) f' = - \phi'.
\end{equation}
Since $f'\geqslant 0$ on $(-\nicefrac L2,0)$, and $\mathcal	G_c '(f) \leqslant 0$, it then follows that  $\phi' \geqslant 0$.
 Therefore, making the choice of $\lambda$ as 
 \begin{equation*}
 	\lambda \bydef \exp(\sup_x \phi(x)),
 \end{equation*}
 we find, for any $x \in (-\nicefrac L2,0)$, that
\begin{equation}\label{h'f'}
	h' =f'+ (\lambda-\mathrm{e}^\phi ) \phi' \geqslant f' \geqslant 0.
\end{equation} 
Let us point out in passing that, since $f$ is even, it follows that $\phi$, and hence $h$, are also even. Moreover, since $f$ is assumed to be non-constant, it is readily seen that $h$ cannot be constant. Indeed, from the relation
\begin{equation*}
	\phi = G_{\lambda} * h,
\end{equation*}
it follows that if $h$ were constant, then $\phi$, and consequently $f$, would also be constant functions.
   Applying now Lemma \ref{parity_elliptic} for $h$, we obtain that
\begin{equation*}
	\phi''(0)<0 \qquad \text{and} \qquad \phi' >0,\quad  \text{on } (-L/2,0).
\end{equation*}
It then follows  from \eqref{diff:main} that 
\begin{equation*}
	f'>0 \qquad \text{and} \qquad \mathcal	G_c'(f) <0, \quad \text{on}  \quad (-L/2,0),
\end{equation*}
which ends the proof of the first part of the lemma. 
\\
In the case when $f \in C^2(\T)$, differentiating \eqref{diff:main} with respect to $x$ yields  
\begin{equation*}
	\mathcal	G_c''(f) (f')^2 +\mathcal	G_c'(f) f'' = - \phi''.
\end{equation*}
Thus, evaluating at $x=0$ implies that 
\begin{equation*}
\mathcal	G_c'(f(0)) f''(0) = - \phi''(0).
\end{equation*}
Now observe that, once the assumption $f \in C^2(\T)$ is made, one immediately deduces that 
\begin{equation*}
	\sup_ {x\in \mathbb T} f(x)=f(0) < a^*(c).
\end{equation*}
Indeed, this follows as a consequence of Proposition \ref{prop:uniform_bound and Holder regularity} as, if $f(0)=a^*(c)$ were to hold, then the exact asymptotic expansion at $x=0$ given by Proposition \ref{prop:uniform_bound and Holder regularity} would contradict the $C^2$ regularity of $f$ at $x=0$.

This said,  we thus have that $\mathcal	G_c'(f(0))<0$, which leads to the bound $f''(0)<0$. Establishing the lower bound $f''(-\nicefrac L2)>0$ can be done similarly. This concludes the proof.
\end{proof}
In the previous section, we discussed local and global regularity properties of solutions to \eqref{main_EQ}. We now show that solutions with half-period monotonicity enjoy improved regularity compared to that of Proposition \ref{prop:uniform_bound and Holder regularity}. The proof relies on the following classical lemma.

\begin{lemma}\label{lemma:W2infty}
Any even function $f$ that belongs to $C^2(\R)\cap \dot  W^{2,\infty}(\R)$ enjoys  the bound 
\begin{equation*}
	\sup_{\substack{x,y \in \R\\ x \neq y}}  \frac{|f(x)-f(y)|}{|x^2-y^2|} \leqslant \tfrac{1}{2} \|f'' \|_{L^\infty (\R)}  . 
\end{equation*}
\end{lemma}

\begin{proof}
	Let $g$ be defined on $[0,\infty)$ by  
	\begin{equation*}
		g(z) \bydef f(\sqrt{z}), \quad z \geqslant 0.
	\end{equation*}
	Note, as $f$ is even and $C^1$ we get  that $f'(0)=0$. Thus, by  l'H\^opital's rule, we obtain that
	\begin{equation*}
	g'(0)=\tfrac12 f''(0), \qquad 	g'(z) = \frac{f'(\sqrt{z})} {2\sqrt{z}}\quad  \text{for any } z > 0. 
	\end{equation*}
Therefore, it follows that
\begin{align*}
\sup_{\substack{x,y \in \R\\ x \neq y}}  \frac{|f(x)-f(y)|}{|x^2-y^2|} &= \sup_{\substack{x,y \geqslant 0 \\ x \neq y}} \frac{|f(x)-f(y)|}{|x^2-y^2|} = \sup_{\substack{t,z \geqslant 0 \\ t \neq z}} \frac{|g(z)-g(t)|}{|z-t|} =\sup_{z \geqslant 0 } |g'(z)| \\
	&\leqslant \tfrac12 \max \left\{f''(0), \sup_{z>0} \tfrac{f'(\sqrt{z})} {\sqrt{z}} \right\} = \tfrac{1}{2} \|f'' \|_{L^\infty (\R)},
\end{align*}
thereby completing the proof of the lemma.
\end{proof}

In our  next result, we improve Proposition \ref{prop:uniform_bound and Holder regularity} by obtaining an estimate in terms of the Lipschitz regularity of the solutions. The asymptotic behavior of the solution $f$ from Proposition \ref{prop:uniform_bound and Holder regularity} will also be improved here. 
 
\begin{proposition}[Lipschitz regularity]\label{pro:Lipregularity}
Let $(c,f)$ be a solution of \eqref{main_EQ} with $f$ being a  continuous, even,  $L$-periodic function that is   non-decreasing on the half-period $(-\nicefrac L2,0)$. Assume further for some $\delta\in(0,1)$ that
\begin{equation*}
0<\delta \leqslant	f(x)\leqslant a^*(c), \quad \text{for all } x\in \mathbb T,
\end{equation*}
where the maximum is attained in the sense that  $f(0)=a^*(c)$. Then, $f$ enjoys the Lipschitz regularity
	\begin{equation*}
 \sup_{x \neq y} \left| \tfrac{f(x)-f(y)}{x-y}\right| \lesssim 1,
	\end{equation*}	
where the Lipschitz constant is independent of the solution $(c,f)$.
Moreover, it holds that 
\begin{equation*}
	a^*(c)-f(x)= \theta |x|+ \mathcal O (|x|^{2}),
\end{equation*}
as $|x|\ll 1$, where $\theta$ is defined in Proposition $\ref{prop:uniform_bound and Holder regularity}$. In addition, if $f$ is non-constant, then $\theta>0$.  
\end{proposition}

\begin{proof}
First, using Proposition \ref{prop:uniform_bound and Holder regularity}, we have that $\|f\|_{C(\T)}$ has a uniform bound that does not depend on the solution $(c,f)$.
Now, since  $f$ is even and $L$-periodic, we may assume without loss of  generality that 
\begin{equation*}
	-\tfrac{L}{2} \leqslant x <y \leqslant 0, \qquad f(x) \neq f(y).
\end{equation*}
Then, we have that 
\begin{equation*}
	0< \delta \leqslant f(x) \leqslant f(\tfrac{x+y}{2}) \leqslant f(y) \leqslant a^*(c).
\end{equation*}
As $\mathcal{G}_c'$ is non-positive and increasing, we find that 
\begin{align*}
\left| \frac{\mathcal{G}_c(f(x)) - 	\mathcal{G}_c(f(y))}{f(x)-f(y)} \right| &=  \left| \int_0^1 \mathcal{G}_c'(t f(x)+(1-t)f(y))\, \ud t \right| 
\\
&\geqslant  \left| \int_0^1 \mathcal{G}_c'(t f(\tfrac{x+y}{2})+(1-t)a^*(c))\, \ud t \right|
\\
&=  \left| \int_0^1 \left( \mathcal{G}_c'(t f(\tfrac{x+y}{2})+(1-t)a^*(c)) - \mathcal{G}_c'(a^*(c))\right) \ud t \right|
\\
&\geqslant \inf_{z \in [\delta, a*(c)]} \mathcal{G}_c''(z) \left|  f(\tfrac{x+y}{2}) - a^*(c)\right| \int_0^1 t\, \ud t 
\\
&\gtrsim a^*(c) -f(\tfrac{x+y}{2}).
\end{align*}
Therefore, it follows that 
\begin{align*}
	\left| \frac{f(x)-f(y)}{x-y} \right| &=	\left| \frac{(f(x)-f(y))(\mathcal{G}_c(f(x)) - 	\mathcal{G}_c(f(y)))}{(\mathcal{G}_c(f(x)) - 	\mathcal{G}_c(f(y)))(x-y)} \right| \left| \frac{\tfrac{x+y}{2}-0}{\frac{x+y}{2}} \right| 
	\\
	&\lesssim \left| \frac{(\phi(x)-\phi(y)) (\frac{x+y}{2}-0)}{ (a^*(c) - f(\frac{x+y}{2}) )(x^2-y^2)} \right| \leqslant \left| \frac{\phi(x)-\phi(y)}{ x^2-y^2} \right|,
	\end{align*}
where we used in the last inequality Proposition \ref{prop:uniform_bound and Holder regularity} with Lemma \ref{lem:strict_monotonicity}.
Using now the fact that $\phi$ is even and belongs to $C^2(\T)$, the Lipschitz regularity follows from Lemma \ref{lemma:W2infty}.

It remains  to prove the asymptotic behavior. To that end, we can simply follow exactly the same lines of the last step of the proof of Proposition \ref{prop:uniform_bound and Holder regularity}, and by employing the Lipschitz regularity instead of the weaker H\"older regularity from Proposition \ref{prop:uniform_bound and Holder regularity}. 
Finally, using Proposition \ref{lem:strict_monotonicity}, we obtain $\phi''(0)=\partial_x^2 \mathcal{H}^{-1}(f)(0)<0$, which implies that the coefficient $\theta$ defined in Proposition \ref{prop:uniform_bound and Holder regularity} is strictly positive, i.e.,   $\theta>0$. This ends the proof of the proposition.
\end{proof}
In the next lemma, we establish a lower bound for the amplitude of the graph of the function $f$ that solves the equation \eqref{main_EQ} for some $c>0$. This will be useful later on in the global bifurcation analysis of \eqref{main_EQ}. More precisely, as will be emphasized in the proof of  Proposition  \ref{thm:GB},  the uniform bound (with respect to $c$) derived in the   next lemma will be crucial to distinguish the limiting profile at the end of the global branch of bifurcation from the  stationary solution. In particular, it prevents the branch from returning to the equilibrium configuration.

\begin{lemma}[Lower bound on the amplitude]\label{lem:amplitude}
Let $(c,f)$ be a solution of \eqref{main_EQ} with $f$ being  an even,  $L$-periodic function that is non-constant and  non-decreasing on the half-period $(-\nicefrac L2,0)$. Assume further that 
\begin{equation*}
	f(0) < a^*(c) \qquad \text{and} \qquad  a^*(c) \geqslant 1.
\end{equation*}
Then, it holds that     
\begin{equation*}
	a^*(c) - f(-L/2) \geqslant  C(p,L)>0,
\end{equation*}
where the constant $C(p,L)$ in the last inequality depends only on $p$ and the period $L$ and is independent on the function $f$ and the speed $c$. 
\end{lemma}

\begin{proof} 
Using Proposition \ref{prop:delta:min} and Proposition \ref{prop:uniform_bound and Holder regularity} we can find two  constants $\delta_0>0 $ and $C >0$, depending only on $p$, such that 
\begin{equation*}
	\delta_0\leqslant f(x) \leqslant C , \quad \text{for all } x\in \mathbb T.
\end{equation*} 
On the one hand, since $f$ is non-constant, then one finds,  in view of Proposition \ref{prop:bound:c}, that  $|c| \leqslant K_{ {\delta_0}}$. Accordingly, it holds that 
\begin{equation*}
 \mathcal G_c''(\xi)= 3 c^2 \xi^{-4}+p''(\xi) \leqslant 3 K_{{\delta_0}}^2 {\delta_0}^{-4}+ \max_{\xi \in [{\delta_0}, a^*(K_{{\delta_0}})]} |p''(\xi)|<\infty,
\end{equation*}
for any $\xi \in [{\delta_0}, a^*(K_{{\delta_0}})]$. This allows us to write, by virtue of Taylor formula, that
\begin{equation*}
	  \left|\mathcal G_c'(a^*(c)) -\mathcal G_c'(f(-L/2))\right| \lesssim _{ {\delta_0}}  |a^*(c)- f(-L/2)|.
\end{equation*} 
On the other hand, one has, by the  monotonicity of $ \mathcal G_c'$ and $f$, that  
	\begin{equation*}
	\mathcal G_c'(f(-L/2)) f'(x) \leqslant \mathcal G_c'(f(x)) f'(x) = -\phi'(x),
	\end{equation*}
	for any $x \in (-L/2,0)$. 
Therefore, using that  $\mathcal G_c'(a^*(c))=0$, and combining the preceding inequalities, we find that 
	\begin{equation*}
\left|a^*(c) -f(-L/2) \right| f'(x)	\gtrsim \left|\mathcal G_c'(a^*(c)) -\mathcal G_c'(f(-L/2))\right| f'(x) \geqslant \phi'(x).
	\end{equation*}
Since $f$ is non-constant, we can choose $L_1,L_2 \in (-\nicefrac L2 ,0)$ such that $L_1<L_2$ and $f(L_1)<f(L_2)$. Integrating the last inequality above over the integral  $[L_1,L_2]$ yields  that 
	\begin{equation}\label{a*-f}
\left(a^*(c) -f(-L/2) \right) \left( f(L_2)-f(L_1) \right) \gtrsim \phi(L_2)-\phi(L_1).
	\end{equation}
With this bound at hand, we now   define
\begin{equation*}
	\lambda \bydef C \geqslant  \exp(\sup_x \phi(x))  \qquad  \text{and} \qquad h \bydef \lambda \phi - \partial_x^2 \phi.
\end{equation*}	
Following \eqref{h'f'}, we can show that $h' \geqslant f'$ which eventually  leads to the inequality 
\begin{equation*}
	h(L_2)-h(L_1) \geqslant f(L_2)-f(L_1).
\end{equation*}
Next, using  Lemma \ref{lem:G}, we obtain that 
\begin{align*}
\phi(L_2)-\phi(L_1) 
&= \int_{L_1}^{L_2} \int_{-L/2}^{L/2} G_{\lambda}(x-y) h'(y)\, \ud y\, \ud x\\
&= \int_{L_1}^{L_2} \int_{-L/2}^{0} \left( G_{\lambda}(x-y) -G_{\lambda}(x+y)  \right) h'(y)\, \ud y \, \ud x\\
& \geqslant \int_{L_1}^{L_2} \int_{L_1}^{L_2} \left( G_{\lambda}(x-y) -G_{\lambda}(x+y)  \right) h'(y) \, \ud y \, \ud x\\
& \gtrsim_{L_1,L_2} h(L_2)-h(L_1)\\ 
& \gtrsim_{L_1,L_2} f(L_2)-f(L_1).
\end{align*}
Summing the last inequality with \eqref{a*-f} and using that $f(L_1)<f(L_2)$, we arrive at the final lower bound 
\begin{equation*}
	a^*(c) - f(-L/2) \gtrsim 1.
\end{equation*}
This completes the proof of the lemma.
\end{proof}

\section{Local Structure of traveling waves}\label{section:local:theory}

This section constitutes the first step in constructing the curve of traveling-wave solutions of \eqref{EP}. We first apply local bifurcation techniques to establish the small-amplitude solutions of \eqref{main_EQ}. We then analyze additional properties of these solutions, which particularly will  come in handy in the analysis of the global continuation of the bifurcation branch in Section \ref{sec:global}

\subsection{Local Bifurcation}\label{sec:localbif}

We introduce the functional  
\begin{equation}\label{nonlinear_problem_2} 
 	\mathcal{F}(c,f)\bydef \mathcal{G}_c(f)+ \mathcal H^{-1}(f) = \tfrac{c^2}{2} \left (\tfrac{1}{f^2} -1\right) + p(f)-p(1) + \mathcal H^{-1}(f) ,
 \end{equation}   
 and  aim to establish a local bifurcation result for the equation 
 \begin{equation*}
 	\mathcal F(c,f)=0,
 \end{equation*}
  in the vicinity of the trivial solution 
  \begin{equation*}
  	(c,f) = (c,1), \quad  c\in \mathbb{R}.
  \end{equation*}
  The local bifurcation theorem, below,    will be stated and proved in the Banach space   
	\begin{equation*}
			X\bydef C_{\mathrm{even}}(\T) = \left \{ f\in C(\mathbb T): f   = \sum_{n \in \mathbb N}h_n \cos \left ( \tfrac {2\pi n }{L} \cdot \right) \right\},
	\end{equation*}
	endowed with the $L^\infty$ norm. However, it is to be emphasized   that the profile $f$ enjoys, in fact, much better regularity properties, which will be discussed with details later on.

\begin{theorem}[Local Bifurcation]\label{thm:1}
There is $\varepsilon >0$, $c_0\in \mathbb{R}$ and  two nontrivial analytic   functions 
\begin{equation*}
	\big(\Psi, \Phi \big) : (-\varepsilon,\varepsilon) \to  \mathbb{R}\times X
\end{equation*}
	such that 
	\begin{equation*}
		\big( \Psi(0) , \Phi (0)\big)  = \left (\sqrt{p'(1) +  \tfrac{1}{1+(\nicefrac {2\pi}{L})^2}},1\right)  
	\end{equation*}
	and 
	\begin{equation*}
		\mathcal{F}\big( \Psi(s) , \Phi (s)\big) = 0, \quad \text{for all } s\in (-\varepsilon,\varepsilon).
	\end{equation*}
	More precisely, there exists an open set $U_0 \subset \mathbb R\times X$, with $(c_0,1) \in U_0$, such that 
	\begin{equation*}
		\left\{ (c,f)\in U_0: \mathcal{F}(c,f)=0, \text{ with }f\not\equiv 1  \right\} = \left\{  \big( \Psi(s) , \Phi (s)\big) , 0< |s| < \varepsilon \right\} ,
	\end{equation*}
	 with  the asymptotic expansions (in the space $X$)
	\begin{equation*}
		\Phi (s)= 1 + s \cos \left( \tfrac{2\pi}{L} \cdot \right) + \mathcal O(s^2),
	\end{equation*}  
	 and  
	\begin{equation*}
		\Psi(s)=  \sqrt{p'(1) +  \tfrac{1}{1+(\nicefrac {2\pi}{L})^2}} +  \Psi'(0) s + \tfrac12 \Psi''(0) s^2 + \mathcal O(s^3),
	\end{equation*}
holding for any $s\in (-\varepsilon,\varepsilon)$, where $\Psi'(0)$ and $\Psi''(0)$ are analyzed in Proposition $\ref{prop:Psi},$ below. 
\end{theorem}

\begin{remark}
It is important to emphasize that the regularity of the profiles $\Psi$ and $ \Phi$ is tied to the  regularity of $p$. Here, we state Theorem \ref{thm:1} under the analyticity assumption on $p$ as this will be required later in order to apply global bifurcation results in the analytic setting. However, local bifurcation branches can be constructed under much weaker regularity assumptions on $p$. In particular, $p \in C^k$ with $k \ge 2$ suffices for the local bifurcation theory.
	\end{remark}
Before turning to the proof of Theorem \ref{thm:1}, we complete its statement with the following proposition.
\begin{proposition}\label{prop:Psi}
	Under the conditions of Theorem $\ref{thm:1},$ we have 
	\begin{equation*}
		\Psi'(0)=0.
	\end{equation*}
		Moreover, for a fixed pressure $P$, there exists a set $\mathcal N \subset (0,\infty)$ of cardinal at most $8$, such that, for any $L \in (0,\infty) \setminus \mathcal N$, we have  
		\begin{equation*}
		\Psi''(0)\neq 0.
	\end{equation*}
\end{proposition}

\begin{proof}[Proof of Theorem $\ref{thm:1}$] 
Our argument is based on applying the analytic version of Crandall--Rabinowitz's theorem to construct the local branch of bifurcation. 
For clarity we split the proof into several steps.

\paragraph{\underline{\em Setup}}  First, we begin by noticing that  
\begin{equation*}
	\mathcal F(c,1)=0, \quad \text{for all } c\in \mathbb R,
\end{equation*}
	that is $f=1$ is a trivial solution for any speed $c\in \mathbb R$. 
\\	
	Next, we consider the open subset 
\begin{equation*}
	U_{\delta } \bydef  \{ (c, f)\in \mathbb R\times  X: \   \delta<  f (x) < \beta (c) \}    ,
\end{equation*}
for a fixed $\delta\in (0,1)$, and a continuous function $c\mapsto \beta (c)> 1$, and we  
 claim that $\mathcal F$ is well defined as a mapping  
\begin{equation}\label{Fmap}
	\mathcal F:  U_{\delta} \to  \big( C(\mathbb T), \norm \cdot_{L^\infty} \big) .
	\end{equation}
	  To that end, we notice, on the one hand, that the local component of $\mathcal F$, that is 
\begin{equation*}
	f\mapsto  \frac{c^2}{2} \left (\frac{1}{f^2} -1\right) + p(f)-p(1), 
\end{equation*}
  maps positive, even continuous functions into $C(\mathbb T)$. 
\\
On the other hand, in view of Proposition \ref{prop:elliptic},   the nonlocal component of $\mathcal F$, that is the operator $\mathcal H^{-1}$,   maps the space of positive, even continuous functions into $C(\mathbb T)$.     This establishes that $\mathcal F$ 
is well defined as previously claimed.

\paragraph{\underline{\em Regularity}} Since $p$ is locally analytic on $(0,\infty)$, then the local component of $\mathcal F$ is locally analytic from $U_\delta$ to $C(\T)$. Consider now the operator $\mathcal{H} : C^2(\T) \to C(\T)$ and let $\phi_0 \in C^2(\T)$ be fixed. Then, it is readily seen that  
\begin{equation*}
	\ud_\phi \mathcal{H}(\phi_0) =(\mathrm{e}^{\phi_0} - \partial_x^2) : C^2(\T) \to C(\T),
\end{equation*} 
is an  isomorphism.   Indeed, one can show, at first, that this map is an isomorphism from $H^2(\mathbb T)$ to $L^2(\mathbb T)$, proving in particular that, for any $v\in C(\mathbb T)\subset L^2(\mathbb T)$,   there exists  $b$   in $H^2(\mathbb T)$ satisfying  
\begin{equation*}
	-\partial_{xx} b+\mathrm  e^{\phi_0}b  = v.
\end{equation*}
Then, one can further show, by reordering the terms in the preceding equation, that $b$ actually belongs  to  $C^2(\mathbb T)$.

 By Proposition \ref{prop:elliptic}, the analytic inverse function theorem (see Theorem 4.5.3 in \cite{BJ03})   and from  the uniqueness of the inverse, we deduce that 
\begin{equation*}
	\mathcal{H}^{-1} : U_\delta \to C^2(\T)
\end{equation*}
is locally analytic. Therefore, the map \eqref{Fmap} is locally analytic.

\paragraph{\underline{\em  Linearized operator: Unidimensional Kernel}}
 Next, we show that the linearized operator has a one-dimensional kernel, i.e.,
 \begin{equation*}
 	\ker \big( \ud_f \mathcal F (c,1) \big) = \langle g\rangle, 
 \end{equation*}
	for some $(c,g)\in \mathbb R \times  X.$ To that end, computing that 
	\begin{equation*}
		\ud_f \mathcal F (c,1)h=  (p'(1) -c^2) h +  (\id - \partial_x^2 )^{-1} h,
	\end{equation*}
	and incorporating the expansion  
	\begin{equation*}
		h(x)= \sum_{n\in \mathbb N} h_n \cos \left ( \tfrac {2\pi n x}{L}  \right) 
	\end{equation*}
	yields that
	\begin{equation*}
		\left (p'(1) - c^2 + \frac{1}{1+(\nicefrac {2\pi n}{L})^2}\right) h_n=0, \quad \text{for all } n\in \mathbb N,
	\end{equation*}
	as soon as $h\in \ker\big(\ud_f \mathcal F (c,1)\big)$. Hence, $h$ would be nontrivial in  the kernel if the dispersion relation 
	\begin{equation*}
		p'(1) - c^2 + \tfrac{1}{1+(\nicefrac {2\pi m}{L})^2}=0
	\end{equation*}
	holds for some $m\in \mathbb N$. Since the period $L>0$ is arbitrary, we take here $m=1$. This can be recast as   
	\begin{equation}\label{c0def}
		c= c_{0}\bydef  \sqrt{p'(1) +  \frac{1}{1+(\nicefrac {2\pi}{L})^2}}\cdot
	\end{equation}
	This formula gives an expression of admissible values of  speeds at which the linearized  operator $\ud_f \mathcal F (c_{0},1) $ has exactly one nontrivial generator of its kernel, which is, modulo a multiplicative constant, given by
	\begin{equation*}
		g(x)= \cos \left ( \tfrac {2\pi x}{L} \right), \quad \text{for all } x\in \mathbb T.
	\end{equation*}

\paragraph{\underline{\em Linearized operator: One co-dimensional Range}} 
	
	It is readily seen, for any $h\in X$, that 
	\begin{equation*}
		\ud_f \mathcal F (c_{0},1) h(\cdot)= \sum_{n\in \mathbb N} \left (  \tfrac{1}{1+(\nicefrac {2\pi n}{L})^2}-\tfrac{1}{1+(\nicefrac {2\pi}{L})^2}\right) h_n \cos \left ( \tfrac {2\pi n \cdot}{L} \right) .
	\end{equation*}
	It then follows that 
	\begin{equation*}
		\text{Range}\big( \ud_f \mathcal F (c_{0},1)\big) \subset A \bydef \left \{ g\in X : \int_{\mathbb T}  g(x)\cos  \left ( \tfrac {2\pi x}{L} \right) \ud x=0  \right\}.
	\end{equation*}
		Conversely, take any  continuous  function $g \in A$ with the expansion 
	\begin{equation*}
		g(\cdot )= \sum_{n\in \mathbb N \setminus \{1\}} g_n \cos\left ( \tfrac {2\pi n \cdot}{L} \right),
	\end{equation*}
	for some  real sequence of  coefficients $(g_n)_{n\in \mathbb N\setminus \{1\}}$, and set 
	\begin{equation*}
		h_n \bydef \frac{g_n}{\left (  \frac{1}{1+(\nicefrac {2\pi n}{L})^2}-\frac{1}{1+(\nicefrac {2\pi}{L})^2}\right)}  , \quad \text{for } n\neq 1,
	\end{equation*}
	with $h_{1}$ being free to take any real value. Accordingly, we set 	\begin{equation*}
		h(x)\bydef \sum_{n\in \mathbb N} h_n \cos\left ( \tfrac {2\pi n x}{L} \right), \quad \text{for any } x\in \mathbb T,
	\end{equation*}
	and we claim that $h\in C(\mathbb T)$. To that end, we write that 
	\begin{equation*}
		\begin{aligned}
			h(x)
			&= h_1\cos \left ( \tfrac {2\pi x}{L} \right)+  \sum_{n\in \mathbb N\setminus\{1\}} (h_n+ (1+(\nicefrac {2\pi}{L})^2)g_n ) \cos\left ( \tfrac {2\pi n x}{L} \right) \\
			&\quad - (1+(\nicefrac {2\pi}{L})^2)\sum_{n\in \mathbb N\setminus\{1\}}  g_n  \cos\left ( \tfrac {2\pi n x}{L} \right) .
		\end{aligned}
	\end{equation*}
	It is then readily seen that the first and the last functions are  continuous. As for the second function, it can further be recast as 
	\begin{equation*}
		\begin{aligned}
			 &\sum_{n\in \mathbb N  \setminus\{1\}} (h_n+ (1+(\nicefrac {2\pi}{L})^2)g_n ) \cos\left ( \frac {2\pi n x}{L} \right)
			 \\
			 &\qquad = (1+(\nicefrac {2\pi}{L})^2)  \sum_{n\in \mathbb N \setminus\{1\}}  \frac{1}{1+(\nicefrac {2\pi n}{L})^2} \left( \frac{g_n}{\frac{1}{1+(\nicefrac {2\pi n}{L})^2}-\frac{1}{1+(\nicefrac {2\pi}{L})^2}}\right)  \cos\left ( \frac {2\pi n x}{L} \right).
		\end{aligned}
	\end{equation*} 
	It is then readily seen that the function above belongs to $H^2(\mathbb T)$ as a consequence of the $\ell^2$-boundedness of the sequence 
	\begin{equation*}
	\left( \frac{(1+n^2)g_n}{\left(1+(\nicefrac {2\pi n}{L})^2\right) \left( \frac{1}{1+(\nicefrac {2\pi n}{L})^2}-\frac{1}{1+(\nicefrac {2\pi}{L})^2}\right)} \right)_{n\in \mathbb N \setminus\{1\}} ,
	\end{equation*}
	which in turn follows from the fact that $g\in C(\mathbb T)\subset L^2(\mathbb T)$.
	Hence, we deduce, by Sobolev embedding, that 
	\begin{equation*}
		x\mapsto  \sum_{n\in \mathbb N} (h_n+ (1+(\nicefrac {2\pi}{L})^2)g_n ) \cos\left ( \tfrac {2\pi n x}{L} \right) \in C(\mathbb T),
	\end{equation*}
	as well. This shows that  $h$ is a continuous function satisfying 
	\begin{equation*}
		\ud_f\mathcal F (c_{0},1) h=g,
	\end{equation*} 
	 whence establishing  that 
	\begin{equation*}
		\text{Range}\big( \ud_f \mathcal F (c_{0},1)\big) = A.
	\end{equation*}
	Consequently, it follows  that $\text{Range}\big( \ud_f \mathcal F (c_{0},1)\big)  $ is of a co-dimension one.

\paragraph{\underline{\em Transversality}} As for the transversality condition, computing that 
	\begin{equation*}
		\partial_c \ud_f \mathcal F (c_{0},1)h= -2c_{0} h,
	\end{equation*}
	for any $h\in X$, it then immediately follows that 
	\begin{equation*}
		\partial_c \ud_f \mathcal F (c_{0},1) \cos \left ( \tfrac {2\pi  \cdot }{L} \right)= -2c_{0} \cos \left ( \tfrac {2\pi  \cdot }{L} \right) \notin A=\text{Range}\big( \ud_f \mathcal F (c_{0},1)\big).
	\end{equation*}
	
	All in all, given the results of our analysis above, we have shown that all the requirements in Crandall--Rabinowitz's theorem are fulfilled, which permits us to apply it and conclude the proof.  
	\end{proof}

\begin{proof}[Proof of Proposition $\ref{prop:Psi}$] 
In order to compute $\Psi'(0)$ and $\Psi''(0)$, we recall that $\mathcal F$ maps $X$ into itself, where 
\begin{equation*}
	X= C_\mathrm{even}(\T).
\end{equation*}
We recall also that 
	\begin{equation*}
	\text{Ker}\big( \ud_f \mathcal F (c_{0},1)\big)	=  \langle \cos (\tfrac{2 \pi}{L} \cdot) \rangle,
	\end{equation*}	
and 
	\begin{equation*}
		\text{Range}\big( \ud_f \mathcal F (c_{0},1)\big) =	A \bydef \left \{ g\in X : \int_{\mathbb T}  g(x)\cos  \left ( \tfrac {2\pi x}{L} \right) \ud x=0  \right\} .
	\end{equation*}
	It is then readily seen that 
	\begin{equation*}
	X= \text{Range}\big( \ud_f \mathcal F (c_{0},1)\big)  \oplus \text{Ker}\big( \ud_f \mathcal F (c_{0},1)\big) =  A \oplus \langle \cos (\tfrac{2 \pi}{L} \cdot) \rangle .
	\end{equation*}
Let $\mathcal P$ be the canonical projection from $X$ into $\langle \xi_0 \rangle$, where 
\begin{equation*}
\xi_0 \bydef \Phi'(0)= \cos (\tfrac{2 \pi}{L} \cdot ).
\end{equation*}
Using  the formulas (I.6.3) and (I.6.11) from \cite{K04156}, we obtain that 
\begin{equation*}
	\Psi'(0)=-\frac{1}{2} \frac{\langle \ud^2_{ff} \mathcal F (c_{0},1) (\xi_0,\xi_0), \xi_0 \rangle}{\langle \ud^2_{c f} (c_0,1) \xi_0,\xi_0 \rangle}.
\end{equation*}
Moreover,  if $\Psi'(0)=0$, then we have tht 
\begin{align}\nonumber
	\Psi''(0)
	&=-\frac{1}{3} \frac{\langle \mathcal{P}\, \ud^3_{fff} \mathcal F (c_{0},1) (\xi_0,\xi_0,\xi_0), \xi_0 \rangle}{\langle \ud^2_{c f} (c_0,1) \xi_0,\xi_0 \rangle}\\ \label{Psi''}
	 &\quad + \frac{\langle \mathcal{P}\, \ud^2_{ff} \mathcal F (c_{0},1) (\xi_0, (\ud_{f} \mathcal F (c_{0},1))^{-1} (\mathrm{Id}-\mathcal{P})\, \ud^2_{ff} \mathcal F (c_{0},1)(\xi_0,\xi_0)), \xi_0 \rangle}{\langle \ud^2_{c f} (c_0,1) \xi_0,\xi_0 \rangle}.
\end{align}
Let us recall that $\phi\bydef \mathcal{H}^{-1}(f)$, and we   define 
\begin{equation*}
	  \mathcal{J}_f \bydef \mathrm{e}^{\mathcal{H}^{-1}(f)}\mathrm{Id}-\partial_x^2 \qquad \text{ and } \qquad \mathcal{J}_1 \bydef \mathrm{Id}-\partial_x^2.
\end{equation*}
Note in passing that  
\begin{equation*}
	\mathcal{J}_1^{-1} \cos(\beta \cdot) = \tfrac{1}{1+\beta^2} \cos(\beta \cdot), \quad \text{for any } \beta \in \R.
\end{equation*}
Considering \eqref{c0def}, we further define  
\begin{equation*}
\alpha \bydef \frac{2\pi}{L}	, \qquad  \mathfrak{L} \bydef \ud_{f} \mathcal F (c_{0},1) = -\tfrac{1}{1+\alpha^2} \mathrm{Id} + \mathcal{J}_1^{-1},
\end{equation*}
and, recalling from \eqref{nonlinear_problem_2}, that 
\begin{equation*} 
 	\mathcal{F}(c,f) = \tfrac{c^2}{2} \left (\tfrac{1}{f^2} -1\right) + p(f)-p(1) + \mathcal H^{-1}(f),
 \end{equation*}   
it then follows by differentiating once that 
\begin{equation*} 
 \ud_f	\mathcal{F}(c,f) h = \left(-c^2f^{-3}+p'(f)\right) h + \mathcal{J}_f^{-1} h,
 \end{equation*} 
and, by differentiating again, that 
\begin{equation*} 
 \ud^2_{ff}	\mathcal{F}(c,f) (h_1,h_2) = \left(3c^2f^{-4}+p''(f)\right) h_1 h_2 - \mathcal{J}_f^{-1} \left( \mathrm{e}^\phi (\mathcal{J}_f^{-1} h_1) (\mathcal{J}_f^{-1} h_2) \right).
 \end{equation*} 
This implies that  
\begin{align} \nonumber
 \ud^2_{ff}	\mathcal{F}(c_0,1) (\xi_0,\xi_0) 
 & = \left(3c_0^2+p''(1)\right)\xi_0^2 - \mathcal{J}_1^{-1} \left( (\mathcal{J}_1^{-1} \xi_0)^2 \right)\\ \nonumber
 & = \left(3c_0^2 +p''(1)\right)\xi_0^2 - \tfrac{1}{2(1+\alpha^2)^2} (1+\tfrac{1}{1+4\alpha^2}\cos (2 \alpha \cdot))\\ \nonumber
 & = \tfrac12 \left(3c_0^2 +p''(1)\right) (1+\cos (2 \alpha \cdot)) - \tfrac{1}{2(1+\alpha^2)^2} (1+\tfrac{1}{1+4\alpha^2}\cos (2 \alpha \cdot))\\ \label{dff}
 & =  \mathscr{A} +  \mathscr{B} \cos (2 \alpha \cdot),
 \end{align} 
 where 
 \begin{equation*}
 	\mathscr A \bydef \tfrac12 \left( 3 c_0^2+  p''(1)\right)-\frac{1}{2(1+\alpha^2)^2} \quad \text{ and } \quad \mathscr B \bydef \tfrac12 \left( 3 c_0^2+  p''(1)\right)-\frac{1}{2(1+\alpha^2)^2(1+4 \alpha^2)}.
 \end{equation*}
Thus, we get  $\Psi'(0)=0$. Differentiating again the main equation, we obtain 
\begin{align*} 
 \ud^3_{fff}	\mathcal{F}(c_0,1) (\xi_0,\xi_0,\xi_0) 
 &= \left(-12c_0^2 + p'''(1)\right) \xi_0^3 - \mathcal{J}_1^{-1} \left( (\mathcal{J}_1^{-1} \xi_0)^3 \right) \\
 &\quad + 3 \mathcal{J}_1^{-1} \left( (\mathcal{J}_1^{-1} \xi_0)\,  \mathcal{J}_1^{-1} (\mathcal{J}_1^{-1} \xi_0)^2\right).
 \end{align*} 
It follows that 
\begin{align}\nonumber  
 \mathcal{P}\, \ud^3_{fff}	\mathcal{F}(c_0,1) (\xi_0,\xi_0,\xi_0) 
 &= \tfrac{3}{4} \left(-12c_0^2 + p'''(1)\right) \xi_0
- \tfrac{3}{4(1+\alpha^2)^4} \xi_0  \\ \nonumber 
 &\quad + \tfrac{3}{(1+\alpha^2)^3} \mathcal{J}_1^{-1} \left( \left(\tfrac12 + \tfrac{1}{2(1+4\alpha^2)} \cos(2 \alpha \cdot ) \right) \xi_0\right)\\  \label{dfff} 
  &= \tfrac{3}{4} \left(-12c_0^2 + p'''(1)\right) \xi_0
+ \tfrac{3}{4(1+\alpha^2)^4} \xi_0 + \tfrac{3}{4(1+\alpha^2)^4 (1+4\alpha^2)}   \xi_0.
 \end{align} 
Noticing now that  one has from \eqref{dff} that
\begin{equation*}
	\begin{aligned}
		\mathfrak{L}^{-1} (\mathrm{Id}-\mathcal{P})\,  \ud^2_{ff}	\mathcal{F}(c_0,1) (\xi_0,\xi_0) 
		&=  (1-\tfrac{1}{1+\alpha^2})  \mathscr{A} +  (\tfrac{1}{1+4\alpha^2}-\tfrac{1}{1+\alpha^2}) \mathscr{B}  \cos (2 \alpha \cdot)
\\
	&\bydef \tilde{\mathscr{A}}  + \tilde{\mathscr{B}}\cos (2 \alpha \cdot),
	\end{aligned}
\end{equation*}
 we deduce that 
\begin{equation*}
	\begin{aligned}
	&\mathcal{P}\, \ud^2_{ff} \mathcal F (c_{0},1) (\xi_0, \mathfrak{L}^{-1}(\mathrm{Id}-\mathcal{P})\, \ud^2_{ff} \mathcal F (c_{0},1)(\xi_0,\xi_0)) \\ 
 &\qquad =\left(3c_0^2+p''(1)\right) \left( \tilde{\mathscr{A}}  \xi_0 +  \tilde{\mathscr{B}}  \mathcal{P}  \xi_0   \cos (2 \alpha \cdot)\right) - \mathcal{P}\mathcal{J}_1^{-1} \left(\tfrac{1}{1+\alpha^2} \xi_0 \left(\tilde{\mathscr{A}} + \tfrac{\tilde{\mathscr{B}}}{1+4\alpha^2} \cos (2\alpha \cdot) \right)  \right)
 \\
 &\qquad =\left(3c_0^2+p''(1)\right) \left( \tilde{\mathscr{A}}   + \tfrac12 \tilde{\mathscr{B}}    \right)\xi_0 -  \tfrac{1}{(1+\alpha^2)^2}  \left(\tilde{\mathscr{A}} + \tfrac{\tilde{\mathscr{B}}}{2(1+4\alpha^2)}  \right) \xi_0
 \\
 &\qquad = ( 2  \mathscr{A} \tilde{\mathscr{A}} + \mathscr{B} \tilde{\mathscr{B}} )\xi_0.
\end{aligned}
\end{equation*}
 Summing up with \eqref{dfff} and using \eqref{Psi''}, we deduce that $\Psi''(0)=0$ if and only if 
 \begin{equation*}
 	2  \mathscr{A} \tilde{\mathscr{A}} + \mathscr{B} \tilde{\mathscr{B}} + \left(3c_0^2 -\tfrac{1}{4} p'''(1)\right)
- \tfrac{1}{4(1+\alpha^2)^4}  - \tfrac{1}{4(1+\alpha^2)^4 (1+4\alpha^2)} =0.
 \end{equation*}
 Multiplying by $4 (1 + \alpha^2)^5( 1 + 4 \alpha^2)^3$, 
this can be seen as a polynomial equation in $\alpha^2$ of degree at most $8$, which is 
\begin{equation*}
\sum_{n=0}^{8} c_n (\alpha^2)^n = 0,
\end{equation*}
where, setting $p_k = p^{(k)}(1)$ for $k \in \{1,2 ,3\}$, the coefficients are given by (the formulas below have been exactly computed using Matlab)
\begin{align*}
a_8 &= 1152p_1^2 + 768p_1p_2 + 128p_2^2 + 768p_1 - 64p_3 \\
a_7 &= 5040p_1^2 + 3360p_1p_2 + 560p_2^2 + 6720p_1 + 768p_2 - 368p_3 + 768 \\
a_6 &= 8640p_1^2 + 5760p_1p_2 + 960p_2^2 + 17712p_1 + 2336p_2 - 892p_3 + 4800 \\
a_5 &= 7191p_1^2 + 4794p_1p_2 + 799p_2^2 + 21564p_1 + 2464p_2 - 1181p_3 + 9024 \\
a_4 &= 2844p_1^2 + 1896p_1p_2 + 316p_2^2 + 13986p_1 + 962p_2 - 925p_3 + 7852 \\
a_3 &= 378p_1^2 + 252p_1p_2 + 42p_2^2 + 5304p_1 + 32p_2 - 434p_3 + 3727 \\
a_2 &= -36p_1^2 - 24p_1p_2 - 4p_2^2 + 1302p_1 - 38p_2 - 118p_3 + 1080 \\
a_1 &= -9p_1^2 - 6p_1p_2 - p_2^2 + 192p_1 - 4p_2 - 17p_3 + 166 \\
a_0 &= 12p_1 - p_3 + 10.
\end{align*}
Analyzing the coefficients of $a_8$, $a_7$ and $a_0$, for instance, we find that they cannot vanish all simultaneously. This ensures that the polynomial cannot be identically zero.
 Therefore, for only at most $8$ values of $\alpha^2$, and thus $L$, we may have that  $\Psi''(0) = 0$. This  concludes the proof of the proposition.  
\end{proof}

\begin{remark}[About the set $\mathcal N$ and the property $\Psi''(0)\neq 0$]\label{remark:about:N:2}
	Observe that the possible vanishing of $\Psi''(0)$ is equivalent to the condition that $L\in \mathcal N$, for some set $\mathcal N$ containing the zeros of the polynomial of order at most eight from the proof above.
	
	 As will become apparent later on, the condition that $\Psi''(0)\not=0$ is merely required to guarantee the validity of the global bifurcation argument, although it is likely not optimal.
One could, in principle, further investigate the situation where $\Psi''(0)=0$ (i.e., when $L$ belongs to the set $\mathcal{N}$ consisting of these eight values) but $\Psi^{(n)}(0)\neq 0$ for some $n \geqslant 3$. However, such analysis would involve highly intricate computations which, we conjecture, might ultimately allow one to recover the admissible values of $L$ within $\mathcal{N}$.

For the sake of simplicity, we choose not to pursue this direction and instead state our main results under a condition that excludes a finite number of values of the period $L$. Nevertheless, we can find   an  example where  $\mathcal N=\varnothing$. Taking, for example, $p(\rho)=\kappa \rho^2$, for some $\kappa>0$, it is readily seen that the coefficients of the polynomial above are all strictly positive (for all values of $\kappa$), except possibly  for $a_1 $ and $a_2$, which are now given by 
\begin{equation*}
	\begin{aligned}
		a_2 &= - 256\kappa^2 + 2528\kappa  + 1080 \\
a_1 &= -64\kappa^2 +  376\kappa    + 166 .
	\end{aligned}
\end{equation*}
Moreover, for $\kappa$ sufficiently small (e.g., $\kappa\leqslant 6$), both $a_1$ and $a_2$ are strictly positive. Consequently, the polynomial with coefficients $a_0,\dots, a_8$ does not vanish  for any $L>0$.
\end{remark} 

\subsection{Properties of small-amplitude traveling waves}

Given the statement of {Theorem \ref{thm:1},} we  now obtain a local branch of solutions $(c_s,f_s)_{s\in (-\varepsilon,\varepsilon)}$ to the nonlinear equation \eqref{nonlinear_problem_2} which satisfies all the properties previously established in Section \ref{sec:localbif}. To proceed, we denote the set corresponding to (the right part of) this curve  by
\begin{equation}\label{Rlocdef}
		\mathcal R_\loc  = \{(c_s, f_s) \in U_{\delta_0}
		 :   \mathcal F(c_s, f_s) =0 , \text{ for all } s\in [0,\varepsilon) \},
	\end{equation}
	where, from now on, the open set $U_{\delta_0}$ is given by 
	\begin{equation}\label{Domain:Def:final}
	U_{\delta_0} \bydef  \{ (c, f)\in \mathbb R\times  X: \   \delta_0 <  f (x) < a^* (c) \}    ,
\end{equation}
	with $\delta_0>0$ is the lower bound given by Proposition \ref{prop:delta:min}, and $a^*(c)$ is the unique solution (in $\xi$) of  the equation 
	\begin{equation*}
		\xi^3 p'(\xi) = c^2.
	\end{equation*} 
	We recall that $a^*(c)$ is the first critical point of the function $\mathcal G_c$, as    previously discussed in Section \ref{section:pressure}. 
The above choice of $\delta_0$ ensures that, along the bifurcation curve, the solutions never touch the side of boundary from below of $U_{\delta_0}$, i.e.,   $f>\delta_0$. 
On the other hand, in view of Propositions \ref{prop:uniform_bound and Holder regularity} and   \ref{pro:Lipregularity}, the choice $\beta(c) = a^*(c)$ ensures that if the maximum  of $f$ is attained exactly at the value $a^*(c)$, then  $f$ is a singular solution to \eqref{main_EQ} (in the sense that it has a limited global Lipschitz regularity).

In the next proposition, we are going to further establish  two additional properties of the local curve of solutions, which will be   useful later  in the construction of the global bifurcation branch.

\begin{proposition}[Additional properties of the local branch of bifurcation]\label{prop:transversality:c}  
For a fixed pressure $P$, there exists $\mathcal N \subset (0,\infty)$ of cardinal at most $8$, such that, for any $L \in (0,\infty) \setminus \mathcal N$ and  $(c_s,f_s)\in \mathcal R_\loc $, the map  
	\begin{equation*}
		s \in (-\varepsilon ,\varepsilon) \mapsto \frac{\ud  c_s  }{\ud s}
	\end{equation*}
is not the zero function.
Moreover, for $\varepsilon$ being taken sufficiently small, it holds that  	 
	\begin{equation*}
		x\mapsto f_s(x) \text{ is non-decreasing on } (-\nicefrac{L}{2} ,0),
	\end{equation*}
	for any $s\in [0 ,\varepsilon).$
\end{proposition}

\begin{proof}
The justification of the first claim follows from the expansion 
\begin{equation*}
	\frac{\ud c_s}{\ud s} = \Psi'(s) = \Psi''(0) s + \mathcal O(s^2), \quad \text{for } |s| < \varepsilon,   \quad  \text{ where }  \Psi''(0)  \neq  0,
\end{equation*}
which is a consequence of Theorem \ref{thm:1} and Proposition \ref{prop:Psi}. We thus focus our attention on the proof of the second claim in the proposition. To that end, we recall  that the solution $f_s$ can be expanded (according to  Theorem \ref{thm:1}, again)  as  
	\begin{equation}\label{local_bif}
		f_s= 1 + s\cos \left ( \tfrac{2\pi  }{L} \cdot \right) + \mathcal O (s^2),  \quad \text{for } |s| < \varepsilon,  
	\end{equation}
where, at the moment,  this asymptotic holds in the  $C(\mathbb T)$-topology. Notice, moreover, that the local bifurcation theory ensures that the map
\begin{equation*}
(-\varepsilon ,\varepsilon) \to  C(\mathbb T), \quad	s\mapsto f_s   , 
\end{equation*} 
is analytic, whence it belongs to $C^2( (-\varepsilon,\varepsilon); C(\mathbb T))$. This, combined with the fact that $f_s$ is a fixed point of the smooth map  
\begin{equation*}
\left\{f \in  C(\mathbb T),\   \delta_c <f(x) < a^*(c) \right\} \to C^2(\mathbb T), \qquad	f \mapsto \mathcal G_{c_s}^{-1} (- \mathcal H^{-1}(f)),
\end{equation*}
yields that 
\begin{equation*}
	s\mapsto f_s \in C^2((-\varepsilon,\varepsilon); C^2(\mathbb T)). 
\end{equation*} 
Therefore, expanding $f_s$, with respect to $s\in (-\varepsilon,\varepsilon)$, entails, by virtue of the uniqueness of Taylor expansion, that \eqref{local_bif} holds in the $C^2(\mathbb T)$-topology. 

To conclude, noticing that $f_s'(0)=f_s'(-\nicefrac{L}{2})=0$, by parity and periodicity, and that the function 
\begin{equation*}
	\cos \left ( \tfrac{2\pi}{L} \cdot \right)
\end{equation*}
is increasing on the interval $(-\nicefrac{L}{2},0)$, we can employ Lemma \ref{lemma:stability:monotonicity} to the perturbation 
\begin{equation*}
	\frac{f_s-1}{s}
\end{equation*}
in order to find that the latter function is increasing on the interval $(-\nicefrac{L}{2},0)$ for any  $s \in (0,\varepsilon)$. This completes the proof of the proposition. 
 \end{proof}

\section{Global bifurcation}\label{sec:global}
Theorem \ref{thm:1} gives rise to the family of solutions 
\begin{equation*}
	(c,f)=(c_s, f_s)_{s\in (-\varepsilon,\varepsilon)}, \quad \text{for some } \varepsilon >0,
\end{equation*}
of the nonlinear equation \eqref{main_EQ}, which recast below for convenience 
\begin{equation*} 
	 \tfrac{c^2}{2} \left (\tfrac{1}{f^2} -1\right) + p(f)-p(1) + \mathcal H^{-1}(f) =0,
\end{equation*}
where 
\begin{equation*}
	(c_0,f_0)= \left (   \sqrt{p'(1) +  \tfrac{1}{1+(\nicefrac {2\pi}{L})^2}}, 1\right).
\end{equation*}
The solution  $(c_s,f_s)$ belongs to the open set $U_{\delta_0}$ defined in \eqref{Domain:Def:final} as  
\begin{equation*}
U_{\delta_0}  = \{ (c,f) \in \mathbb R \times C(\mathbb T)  : \delta_0 < f(x) < a^*(c) \},	
\end{equation*}
for any $s\in (-\varepsilon,\varepsilon)$. We also recall that the choices of $\delta_0$ and $a^*(c)$ are made according to the discussion just after \eqref{Domain:Def:final}.

The purpose of this section is to study the extension of the local bifurcation branch $(c_s, f_s)_{s\in (-\varepsilon,\varepsilon)}$ beyond the interval $s\in (-\varepsilon,\varepsilon)$. More precisely, we will   show that this extends to global branch of bifurcation $(c_s,f_s)_{s\in \mathbb R}$, and we will also    identify many quantitative properties of an ultimate profile, denoted by $f_{\infty}$, which is an object that lives at the end of the global branch of bifurcation. 
From now on,  we will only restrict our analysis to the bifurcation curve corresponding to  $s\geqslant 0$, which we previously agreed to denote it by $\mathcal R_\loc $.

\subsection{Global curve of bifurcation: An existence result}

Let $\mathcal R_\loc $ denote the (right) local branch of bifurcation previously constructed in Theorem \ref{thm:1}, i.e., 
	\begin{equation*}
		\mathcal R_\loc  = \{(c_s, f_s) \in U_{\delta_0}
		 :   \mathcal F(c_s, f_s) =0 , \text{ for all } s\in [0,\varepsilon) \}.
	\end{equation*}
	The following statement ensures the existence of a global (maximal) extension of this curve  of solutions.

\begin{theorem}\label{thm:GB}
For a fixed pressure $P$, there exists $\mathcal N \subset (0,\infty)$ of cardinal at most $8$ such that, for any $L \in (0,\infty) \setminus \mathcal N$,
	there exists a set $\mathcal R$ extending $\mathcal R_\loc $ in the following sense:
	\begin{equation*}
			\mathcal R_\loc \subset \mathcal R  \bydef  \{(c_s, f_s) \in  U_{\delta_0}
		 :   \mathcal F(c_s, f_s) =0 , \text{ for all } s\in [0,\infty)\},
	\end{equation*}
	where the mapping  $ s\mapsto (c_s,f_s) $ is continuous on $[0,\infty)$ and, at each point, $\mathcal R$ has a local analytic reparameterization. 
	
	Moreover, one of the following scenarios occurs: 
	\begin{enumerate}
		\item [(i)] Non-existence of a limiting speed/wave profile, i.e., 
		 \begin{equation*}
		 	\lim _{s\to \infty} \left( |c_s| +  \norm {f_s}_{C(\mathbb T)}\right) = \infty.
		 \end{equation*}
	    \item [(ii)] The global branch approaches the boundary of $ U_{\delta_0}$ as $s$ tends to infinity.
	    \item [(iii)] The global branch is a loop, i.e., there is $T>0$ such that 
	    \begin{equation*}
	    	(c_{s+T},f_{s+T}) = (c_s,f_s), \quad \text{for all } s\geqslant 0,
	    \end{equation*}
	    and, thus, 
	     \begin{equation*}
	     	\mathcal R = \{ (c_{s},f_{s}), \  s\in [0,T]  \}  \quad \text{ and } \quad  (c_{T},f_{T}) = \left(  \sqrt{p'(1) +  \tfrac{1}{1+(\nicefrac {2\pi}{L})^2}},1\right ).
	     \end{equation*} 
	\end{enumerate}
	Furthermore,  for any $\tilde{s} \in (0,\infty)$, there exists a continuous and injective parameterization $\varphi :(-1,1) \to \R$ such that $\varphi(0)=\tilde{s}$ and the map $t \mapsto (c_{\varphi(t)},f_{\varphi(t)})$ is analytic.
\end{theorem}

\begin{remark}
	In the statement above, we only state the properties of the global branch of bifurcation that are directly related to the scope of this paper. Nevertheless, one can further say a lot more about the continuation of the local branch of bifurcation. See   the general global bifurcation statement from  \cite[Theorem 9.1.1]{BJ03} for more details. 
\end{remark}

Before presenting the proof of Theorem \ref{thm:GB}, we are going to first establish two important properties of the curve of solutions to the equation \eqref{main_EQ}. The first result (given in Lemma \ref{lem:Fredholm}, below) shows that the operator $\ud_f \mathcal{F}(c,f)$ remains a Fredholm operator along the bifurcation curve.
The second one (given in Lemma \ref{lem:compactness:brach}, thereafter), discusses  a topological information encoded in the set of solutions to the equation \eqref{main_EQ}. These two lemmas  will constitute  the cornerstones in the proof of  Theorem \ref{thm:GB}.  

\begin{lemma}\label{lem:Fredholm}
	The operator $\ud_f \mathcal{F}(c,f) : C(\T) \to C(\T)$ is a Fredholm operator of index zero for any $(c,f) \in U_{\delta_0}$. 
\end{lemma}

\begin{proof}By a direct computation, one can show that 
\begin{equation*}
	\ud_f \mathcal{F}(c,f) = (p'(f)-c^2f^{-3})\id  + \left(\mathrm{e}^{\mathcal{H}^{-1}(f)} \id   - \partial_x^2\right)^{-1}.
\end{equation*}
On the one hand, we have by virtue of the assumptions on the  pressure from \eqref{p_conditions1}, for any $(c,f) \in U_{\delta_0}$,  that  
\begin{equation*}
	f^3p'(f) < a^*(c)^3 p'(a^*(c))=c^2, 
\end{equation*}
which implies that the map 
\begin{equation*}
(p'(f)-c^2f^{-3}) \id  : C(\T) \to C(\T)
\end{equation*}
is an isomorphism. On the other  hand, it is readily seen that the operator 
\begin{equation*}
\left(\mathrm{e}^{\mathcal{H}^{-1}(f)} \id  - \partial_x^2\right)^{-1}
\end{equation*} 
is compact on $C(\T)$. This shows that $\ud_f \mathcal{F}(c,f) $ is a compact perturbation of an isomorphism, thereby deducing that it is a Fredholm operator of index zero. This completes the proof of the lemma.
\end{proof}

In the next lemma, we establish a compactness property for the set of solutions given by
 \begin{equation*}
		\mathcal S\bydef \{ (c,f) \in U_{\delta_0}: \mathcal F (c,f) =0  \}.
	\end{equation*}

\begin{lemma}[Compactness along the branch of bifurcation]\label{lem:compactness:brach}
Any bounded closed subset of $\mathcal S$ is compact.
\end{lemma}

\begin{proof}  
Let $B$ be a bounded set of $\mathcal{S}$. Then, there exists $\gamma> 0$ such that any $(c,f) \in B$ satisfies  
\begin{equation*}
	|c|\leqslant \gamma , \qquad \delta_0 < f(x) \leqslant a^*(c) \leqslant a^*(\gamma) , 
\end{equation*}
for any $x \in \T$, where we used the fact that $a^*$ is increasing (see Section \ref{section:pressure}).
Therefore, employing Proposition \ref{prop:uniform_bound and Holder regularity}, we find that $B$ is bounded in $\R \times C^{\frac12}(\T)$, whence relatively compact in $\R \times C(\T)$. This completes the proof of the lemma.
\end{proof}

\begin{proof}[Proof of Theorem $\ref{thm:GB}$] With Proposition \ref{prop:transversality:c}, Lemmas \ref{lem:Fredholm} and   \ref{lem:compactness:brach} in hand, it is readily seen that Theorem \ref{thm:GB}    follows by applying the global bifurcation theorem (Theorem \ref{thm:global:CR}).  \end{proof}

\subsection{Global curve of bifurcation: Discussion of the potential scenarios}

Now, we are in a position to examine the scenarios \textit{(i)}, \textit{(ii)} and \textit{(iii)} given by Theorem \ref{thm:GB}. We will show that both \textit{(i)} and \textit{(iii)} cannot occur, thereby forcing  \textit{(ii)} to hold.
\\
We begin by ruling out   scenario \textit{(iii)}. For this purpose, we  introduce the fluctuation profile
\begin{equation*}
	\tilde{f} \bydef f-1.
\end{equation*}
Accordingly, we define the shifted sets
\begin{equation*}
	\tilde{\mathcal{S}} \bydef \{ (c,\tilde{f}), \  (c,\tilde{f}+1) \in  \mathcal{S}\}
	 \qquad 	\text{ and } \qquad \tilde{\mathcal{R}} \bydef \{ (c,\tilde{f}),  \ (c,\tilde{f}+1) \in  \mathcal{R}\}.
\end{equation*}
We further  introduce   the set 
\begin{equation}\label{cone:def}
	\mathcal K\bydef  \big  \{ \tilde{f} \in C_{\mathrm {even}}(\mathbb T) : \tilde{f} \text{ is non-decreasing on  } (-\nicefrac{L}{2} ,0) \text{ and } \tilde{f}(-\nicefrac{L}{2}) \tilde{f}(0) \leqslant 0 \big \},
\end{equation}
which is obviously a cone in the sense that it is stable under the multiplication by non-negative real constants. Finally, we introduce  the subset of $\tilde{\mathcal{S}}$
\begin{equation*}
	\mathcal{T} \bydef \tilde{\mathcal{S}} \setminus \{0\}.
\end{equation*}
We now prove the following proposition.

\begin{proposition}\label{prop:openness:cone}  
	Each point of  $ (\R \times \mathcal K) \cap \tilde{\mathcal R} \cap \mathcal{T} $   is an interior point of $(\R \times \mathcal K) \cap \mathcal{T} $ in $\tilde{\mathcal S}$.
\end{proposition}

\begin{proof}
We proceed in two steps.
\\
\underline{\em Step 1. From $C^0$-stability to $C^2$-Stability.} Let $(c_f,f) \in \mathcal{S}$ be fixed, and consider another solution $(c_g,g) \in \mathcal{S}$ such that
	\begin{equation}\label{stability:bound:cf}
		|c_f-c_g|+\norm {f-g}_{C(\mathbb T)} \leqslant \varepsilon,
	\end{equation}
	for $\varepsilon>0$ small enough in a sense to be made precise, later on. A direct consequence, by virtue of Proposition \ref{prop:elliptic}, is that 
	\begin{equation}\label{stability:bound:phi}
		\norm {\phi_f-\phi_g}_{C^2(\T)}\lesssim \varepsilon,
	\end{equation}
	where we denote 
	\begin{equation*}
		\phi_f = \mathcal H^{-1}(f), \quad \phi_g = \mathcal H^{-1}(g).
	\end{equation*}
	Indeed, this can be shown by, first, employing Sobolev embedding and  the  stability estimate \eqref{monotonicity:H} from the proof of Proposition \ref{prop:elliptic}, which ensure that 
	\begin{equation*}
		\norm {\phi_f-\phi_g}_{C(\T)} \lesssim \norm {\phi_f-\phi_g}_{H^1(\T)} \lesssim_{\delta_0} \norm {f-g}_{L^2(\T)}\lesssim \varepsilon.
	\end{equation*}
	Then, recalling that 
	\begin{equation*}
		\partial_{x}^2 \phi_f - \partial_{x}^2\phi_g = \mathrm{e}^{\phi_f}  - \mathrm{e}^{\phi_g} + g-f, 
	\end{equation*}
	yields eventually to the control 
	\begin{equation*}
		\norm {\phi_f-\phi_g}_{C^2(\T)} \lesssim \norm {\phi_f-\phi_g}_{C(\T)}+  \norm {f-g}_{C(\T)}\lesssim \varepsilon.
	\end{equation*} 
	 Now, as $(c_f,f) $ and $ (c_g,g)$ are both in $ U_{\delta_0}$, we have that
	\begin{equation}\label{fga^*}
		\max_{x\in \mathbb T} f(x)< a^*(c_f) \qquad \text{and} \qquad \max_{x\in \mathbb T} g(x)<  a^*(c_g).
	\end{equation}
	Thus, by Proposition \ref{prop:smoothness:1}, both $f$ and $g$ are smooth ($f,g\in C^\infty (\mathbb T)$). One can in fact show, in that case, that 
	\begin{equation*}
		\norm {(f,g)}_{C^2(\mathbb T)}\lesssim_{\delta_0} \norm {(f,g)}_{C(\mathbb T)}.
	\end{equation*} 
Using next the equation \eqref{main_EQ} to write	    
		\begin{equation*}
		\mathcal G_{c_f}(f)-\mathcal  G_{c_g}(g) =  \phi_g -\phi_f,
		\end{equation*}
	and it follows that  
		\begin{equation*}
		\mathcal G_{c_f}(f)-\mathcal  G_{c_f}(g) = \mathcal  G_{c_g}(g) -\mathcal  G_{c_f}(g) +   \phi_g -\phi_f= \tfrac12 (c_g^2-c_f^2)(g^{-2}-1)+   \phi_g -\phi_f.
		\end{equation*}
Since $g \geqslant \delta_0$ and $g$ is smooth,  it then happens that 
$$\|g^{-2}-1\|_{C^2(\T)} \lesssim_{\delta_0} 1.$$ Therefore, using \eqref{stability:bound:cf} with \eqref{stability:bound:phi}, we obtain that
		\begin{equation*}
		\left\|\mathcal G_{c_f}(f)-\mathcal  G_{c_f}(g)\right\|_{C^2(\T)} \lesssim_{\delta_0} \varepsilon.
		\end{equation*}	
In view of \eqref{stability:bound:cf} and \eqref{fga^*}, taking $\varepsilon>0$ small enough, one finds that  
	\begin{equation*}
		\max_{x\in \mathbb T} g(x)< a^*(c_f).
	\end{equation*}	
	With this bound, together with the fact that   the map $\mathcal{G}_{c_f}$ is invertible on $[\delta_0,\Delta_0]$, with a smooth inverse,  where 
\begin{equation*}
\Delta_0 \bydef  \max\left\{\max_{x\in \mathbb T} f(x), \max_{x\in \mathbb T} g(x) \right\} < 	a^*(c_f),
\end{equation*}
we arrive at the bound  
	\begin{align*}
\left\| f-g\right\|_{C^2(\T)} 
&= \left\|\mathcal G_{c_f}^{-1}(\mathcal G_{c_f}(f))-\mathcal G_{c_f}^{-1}(\mathcal  G_{c_f}(g))\right\|_{C^2(\T)}  \lesssim_{\delta_0,c_f,f}	\left\|\mathcal G_{c_f}(f)-\mathcal  G_{c_f}(g)\right\|_{C^2(\T)} \lesssim_{\delta_0} \varepsilon.
		\end{align*}	
\underline{\em Step 2. Conclusion.} 
Let now  $(c_f,\tilde{f}) \in  (\R \times \mathcal K) \cap \tilde{\mathcal R} \cap \mathcal{T} $ be fixed,  and consider  $(c_g,\tilde{g}) \in \tilde{\mathcal S}$ being sufficiently close to $(c_f,\tilde{f})$ in $\R \times C(\T)$. Recall that, taking
\begin{equation*}
	f=\tilde{f}+1, \qquad g=\tilde{g}+1,
\end{equation*}
translates into $(c_f,f),(c_g,g) \in \mathcal{S}$ and, by virtue of  the previous step, we find that $f$ and $g$ are   close to each other  in the $C^2(\T)$ topology, as well.

Notice that, from the definition of the cone \eqref{cone:def}, the only constant function in $\mathcal{K}$ is the zero function. Since $(c_f,\tilde{f}) \in \mathcal{T}$, then $\tilde{f}$ is not constant. Therefore, $f$ satisfies the conditions of  Lemma \ref{lem:strict_monotonicity}, which allows us to deduce that 
\begin{equation*}
	f' > 0 \quad \text{ on } (-\nicefrac{L}{2},0) \qquad \text{and }  \qquad f''(0)<0<f''(- \nicefrac{L}{2}).
\end{equation*}
Now, since $g$ is even and $L$-periodic,  then $g'(0)=g'(-\nicefrac{L}{2})=0$. Therefore, we deduce, by  Lemma \ref{lemma:stability:monotonicity}, that $g$ (as a sufficiently small $C^2$-perturbation of $f$ with the preceding properties) is non-decreasing on $(-\nicefrac{L}{2},0)$ as well.
\\
To conclude, we employ the fact that  $\tilde{f}(-\nicefrac{L}{2}) \tilde{f}(0) \leqslant 0$ to deduce that $f(x_0)=1$, for some $x_0\in \T$. Thus, recalling that $f \not\equiv 1$, Lemma \ref{lamma:sign} ensures that  
\begin{equation*}
	   \min_{x\in\T} f(x) < 1 < \max_{x \in \T} f(x) .
\end{equation*}
Accordingly, one can deduce, by taking   $\varepsilon>0$ small enough,   that 
$$\tilde{g}(-\nicefrac{L}{2}) \tilde{g}(0) \leqslant 0,$$ which, eventually, yields  that $\tilde{g} \in \mathcal{K}$. This completes the proof of the proposition. 
\end{proof}

\begin{proposition}[Excluding finitely-periodic loops]\label{prop:iii:fails}
	 In the notation of Theorem $\ref{thm:GB},$ the scenario (iii) does not occur and $f_s -1 \in \mathcal K \setminus \{0\}$ for any $s \in (0,\infty)$.
\end{proposition}

\begin{proof} 
Our argument  is grounded on applying Theorem \ref{thm:global:cone} about the global bifurcation in cones to exclude scenario (iii) from the statement of Theorem \ref{thm:GB}. To that end, consider the cone  $\mathcal K$ given by \eqref{cone:def}, and we define
	\begin{equation*}
		\tilde{\mathcal R}_\loc \bydef \{(c,\tilde{f}), (c,\tilde{f}+1) \in  \mathcal R_\loc \},
	\end{equation*}		
where $\mathcal R_\loc$ is defined in \eqref{Rlocdef}.
By virtue of Proposition \ref{prop:transversality:c}, together with \eqref{local_bif}, we have that 	
\begin{equation*}
		\tilde{\mathcal R}_\loc \subset \mathbb R\times \mathcal K,
\end{equation*}	
for small $\varepsilon>0$, which establishes the validity of condition ($i$) from  Theorem \ref{thm:global:cone}. 
Moreover,  Proposition \ref{prop:openness:cone} establishes  condition ($iii$) from   Theorem \ref{thm:global:cone}, as well. 
	\\
We are thus only left with verifying the validity of   condition ($ii$) from Theorem \ref{thm:global:cone}, which requires to show that any   couple of elements $(c,\xi)\in \mathbb R\times C(\mathbb T)$ such that  
	\begin{equation*}
	\xi \in \ker (\ud_f\mathcal F (c,1))\cap \mathcal K	\setminus \{0\}
	\end{equation*}
	is uniquely determined by 
	\begin{equation*} 
		(c,\xi)= (c_0, \tau \xi_0), \quad \text{for some } \tau > 0,
	\end{equation*}
	where 
	\begin{equation}\label{c_0xi_0}
		c_0 =  \sqrt{p'(1)+\tfrac{1}{1+(\nicefrac {2\pi}{L})^2}} \qquad \text{ and } \qquad \xi_0 = \cos (\tfrac{2\pi }{L}\cdot)
	\end{equation}
	are the components of the bifurcation speed and the generator of the kernel associated with the linearized operator, previously established in the local bifurcation theory (see the proof of Theorem \ref{thm:1}). The point here is that, being in the cone $\mathcal K$ selects the sign of the coefficient $\tau$ in the above. Indeed, let $(c,\xi)\in \mathbb R \times C(\mathbb T)$ with  $	\xi \in \ker (\ud_f\mathcal F (c,1))\cap \mathcal K	\setminus \{0\}$. Then, by reproducing the spectral analysis of the linearized operator previously laid out in the proof of Theorem \ref{thm:1}, that 
	\begin{equation*}
		\left (p'(1) - c^2 + \frac{1}{1+(\nicefrac {2\pi k}{L})^2}\right) \xi_k=0, \quad \text{for all } k\in \mathbb N,
	\end{equation*}
	where 
	\begin{equation*}
		\xi(x) = \sum_{k\geqslant 0} \xi_k \cos \left (\tfrac{2\pi k x}{L}\right), \quad (\xi_k)_{k\in \mathbb N}\subset \mathbb R.
	\end{equation*}
	Since $\xi \neq 0$, then there exists $\tilde{k} \in \N $, such that 
	\begin{equation*}
				\xi(x) = \xi_{\tilde{k}} \cos (\tfrac{2\pi \tilde{k} x}{L}), \qquad c^2= p'(1)  + \frac{1}{1+ (\nicefrac {2\pi \tilde{k}}{L})^2}.
	\end{equation*}
As  $\xi$ must be non-decreasing on $(-\nicefrac{L}{2},0)$ as an element of the cone $ \mathcal K$, then we should have that $\tilde{k} = 1$ and  $\xi_{\tilde k} > 0$, thereby concluding the proof of \eqref{c_0xi_0}. The proof of the proposition is now completed. 
\end{proof}

\begin{proposition}[Excluding blow-up]\label{prop:i:fails}
	 In the notation of Theorem $\ref{thm:GB},$ the scenario (i) cannot   occur.
\end{proposition}

\begin{proof}
Having Propositions \ref{prop:bound:c} and \ref{prop:uniform_bound and Holder regularity} in hand, we deduce that $(f_s)_{s \geqslant 0}$ is uniformly bounded in $C(\T)$.
	On the other hand,   Proposition \ref{prop:iii:fails}, ensures that  $f_s$ does not return the equilibrium ``$1$'' as soon as $s\in  (0,\infty)$. Therefore, using Proposition \ref{prop:bound:c}, anew, we obtain that $|c_s| \leqslant K_{p,\delta}$ for any $s \geqslant0$. This concludes the proof the boundedness of $(c_s,f_s)_{s\in [0,\infty)}$ in $\mathbb R \times C(\mathbb T)$.
\end{proof}

\subsection{Singular waves}

From Proposition \ref{prop:iii:fails} and Proposition \ref{prop:i:fails}, we deduce that the scenario $(ii)$ from Theorem \ref{thm:GB} must occur. We prove in  what follows that the limiting profile  of the global bifurcation curve is a singular solution to \eqref{main_EQ} in the sense that it has an exact global Lipschitz regularity.

\begin{proposition}[Touching the boundary at the maximum]\label{prop:exit} 
Up to a subsequence,  the family of solutions $(c_s,f_s)_{s \in [0,\infty)}$ converges, as $s \to \infty$, to some couple $(c_\infty,f_\infty)$ in $\R \times C(\T)$. Moreover, $f_\infty$ is even, Lipschitz, increasing on $(-\nicefrac{L}{2},0)$ and of class $C^\infty$ on $(-\nicefrac{L}{2},0)$. In addition, it holds that   
\begin{equation*}
	0 < \delta_0 < f_\infty(-\nicefrac{L}{2})<1<f_\infty(0)=c^*(c_\infty).
\end{equation*} 
\end{proposition}

\begin{proof}
A direct consequence of Theorem \ref{thm:GB}, Propositions \ref{prop:iii:fails} and   \ref{prop:i:fails} is that $(c_s,f_s)_{s \in [0,\infty)}$ approaches the boundary of $U_{\delta_0}$, i.e., either 
\begin{equation*}
	\delta_0 = \min_{x \in \T} f_\infty(x) \qquad \text{or} \qquad a^*(c_\infty) = \max_{x \in \T} f_\infty(x).
\end{equation*}
Since $f_s$ is non-decreasing on $(-\nicefrac{L}{2},0)$, for any $s\in[0,\infty)$, then $f_\infty$ is non-decreasing as well. Also, we know from Proposition \ref{prop:delta:min}, that $f_\infty(x) \geqslant 2 \delta_0$, which insures the limiting profile $f_\infty$ touches the boundary of $U_{\delta}$ from above, i.e., that 
\begin{equation*}
	f_\infty(0) = a^*(c_\infty).
\end{equation*}
Using Lemma \ref{lem:amplitude}, we obtain, for any $s \in (0,\infty)$, that 
\begin{equation*}
	a^*(c_s) - f_s(-\nicefrac{L}{2}) \gtrsim_{L} 1,
\end{equation*}
where the constant in the last inequality is independent of $s$, $f_s$ and $c_s$. 
Taking the limit $s \to \infty$ (up to a subsequence), we deduce that  
\begin{equation*}
f_\infty(0)- f_\infty(-\nicefrac{L}{2})  =	a^*(c_\infty) - f_\infty(-\nicefrac{L}{2}) > 0,
\end{equation*}
which implies that $f_\infty$ is not a constant function. 
From the definition of the cone \eqref{cone:def}, and since $f_s-1 \in \mathcal{K}$, for any $s\in (0,\infty)$, we deduce that 
\begin{equation*}
f_\infty(-\nicefrac{L}{2}) \leqslant 1 \leqslant	f_\infty(0). 
\end{equation*}
The inequalities above are in fact strict due to Lemma \ref{lamma:sign}.
Further using Propositions \ref{prop:smoothness:1} and   \ref{pro:Lipregularity}, we obtain the $C^\infty$ regularity away from the maximal value and the global Lipschitz regularity.
Finally, the strict monotonicity on the half period follows from Lemma \ref{lem:strict_monotonicity}. This completes the proof of the proposition.
\end{proof}

We are now in a position to conclude the discussion of the global bifurcation analysis. For the sake of completeness, we summarize by establishing the main results of the paper, as a consequence of the previous findings from this section.

\begin{proof}[Proof of Theorem $\ref{thm:BifCurv}$] 
The existence of the global curve of bifurcation, together with most of its properties, follow directly  from Theorem \ref{thm:GB}. Moreover, the fact that each $f_s$, for any $s\in (0,\infty)$, is increasing on the half-period is a consequence of Proposition \ref{prop:iii:fails} and Lemma \ref{lem:strict_monotonicity}.
\end{proof}

\begin{proof}[Proof of Theorem $\ref{thm:LimitObject}$]
	This  follows from Propositions \ref{pro:Lipregularity} and \ref{prop:exit}.
\end{proof}

\appendix
\section{Bifurcation Toolkit} For convenience, we collect in this section the core fundamental theorems on local and global bifurcation that we employ throughout the paper. 

The first theorem is due to Crandall and Rabinowitz \cite{CR71}. The version of the statement below is taken from \cite{BJ03}.
\begin{theorem}[Crandall--Rabinowitz]\label{thm:CR} 
	Let $X$ and $Y$ be Banach spaces, and $\mathcal F: \mathbb R\times X\to Y$ be of class $C^k$, for some $k\geqslant 2$, satisfying $\mathcal F(\lambda,x_0)=0$,  for some $x_0\in X$ and all $\lambda\in \mathbb R$. Assume further, for some $\lambda_0\in \mathbb R$, that 
	\begin{enumerate}[label=$(\roman*)$]
		\item The linearized operator $\mathcal L\bydef \partial_x \mathcal F[(\lambda_0,x_0)]$ is Fredholm of index zero.
		\item The kernel of $L$ is one dimensional, generated by some element $\xi_0\in X$.
		\item The transversality condition
		\begin{equation*}
			 \partial^2_{\lambda,x} \mathcal F[(\lambda_0,x_0)] (1,\xi_0) \notin Range (\mathcal L)
		\end{equation*}
		holds. 
	\end{enumerate}
	Then, $(\lambda_0,x_0)$ is a bifurcation point. More precisely, there exists $\varepsilon>0$ and a branch of solutions given by 
	\begin{equation*}
		\left \{ (\lambda,x)= (\Psi(s),\Phi(s)) \mid s\in I_\varepsilon\bydef  (-\varepsilon,\varepsilon)\right\} \subset \mathbb R \times X ,
	\end{equation*}
	where $\Psi,\Phi\in C^{k-1}(I_\varepsilon)$ and satify that 
	\begin{equation*}
		\Psi(0)= \lambda_0 \qquad \text{and} \qquad  \Phi'(0)=\xi_0.
	\end{equation*}
	Moreover, 
	there exists an open set $U_0 \subset \mathbb R\times X$, with $(\lambda_0,x_0) \in U_0$, such that 
	\begin{equation*}
		\left \{   (\lambda,x)\in U_0 \mid  \mathcal F(\lambda,x)=0 \text{ and } x\not\equiv x_0\right\}=\left \{  (\Psi(s),\Phi(s)) \mid s\in I_\varepsilon \text{ and }s \neq 0\right\}  .
	\end{equation*}
	 If, furthermore, $\mathcal F$ is analytic, then $\Psi$ and $\Phi$ are analytic functions on $I_\varepsilon$. 
\end{theorem} 
 
From now on, in view of the notations from the preceding theorem, we further assume that $\mathcal F$ is analytic on some open set $U\subset \mathbb R\times X$. The extension of a local branch of bifurcation is based on the following items: 
\begin{itemize}
	\item[(A1)] For any $\lambda\in \mathbb R$, it holds that 
	 \begin{equation*}
	 	(\lambda,x_0)\in U \qquad \text{and} \qquad \mathcal F(\lambda,x_0)=0.
	 \end{equation*}
	 \item[(A2)] For any $(\lambda,x)\in \mathbb R\times X$ such that $ \mathcal F(\lambda,x)=0$, it holds that $ \partial_x\mathcal F[(\lambda,x)]$ is a Fredholm operator of index zero.
	 \item [(A2)] For some $\lambda_0\in \mathbb R$ and $\xi_0\in X$, it holds that 
	 \begin{equation*}
	 	\ker \left( \partial_x\mathcal F[(\lambda_0,x_0)] \right) =\{s\xi_0 \mid s\in \mathbb R \},
	 \end{equation*}
	 and
	 \begin{equation*}
	 	\partial^2_{\lambda,x} \mathcal F[(\lambda_0,x_0)] (1,\xi_0) \notin Range \left( \partial_x\mathcal F[(\lambda_0,x_0)] \right) .
	 \end{equation*}
\end{itemize}
Given the above assumptions, and based on the (analytic) local branch of bifurcation constructed in Theorem \ref{thm:CR}, we now introduce the set corresponding to the (strict-right) local curve of bifurcation 
\begin{equation*}
	\mathcal R^+\bydef \left \{  (\Psi(s),\Phi(s)) \mid s\in (0,\varepsilon)\right\},
\end{equation*}
the set of all zeros of $\mathcal F$ in the open set $U$
\begin{equation*}
	\mathcal S\bydef \left \{  (\lambda,x)\in U \mid \mathcal F(\lambda,x)=0\right\}
\end{equation*}
and the set of non-trivial zeros of $\mathcal F$ in the open set $U$
\begin{equation*}
	\mathcal T \bydef \left \{  (\lambda,x)\in \mathcal S \mid x\neq x_0\right\}.
\end{equation*}
Here, we agree that $\varepsilon$ is being taken sufficiently small in a way that 
\begin{equation*}
	\mathcal R^+\subset \mathcal T \quad \text{ and } \quad \Phi'(s)\neq 0,
\end{equation*}
for all $s\in (-\varepsilon,\varepsilon)$.

The next theorem is about the existence of a global branch of bifurcation, extending $\mathcal R^+$. Its statement is taken from \cite[Theorem 9.1.1]{BJ03}. 
\begin{theorem}[Global one-dimensional branches]\label{thm:global:CR}
	Assume the validity of (A1)--(A3) above, and suppose further that 	\begin{enumerate}[label=$(\roman*)$]
		\item Transversality in the direction of $\Psi$: $ \Psi'\not\equiv 0$ on $(-\varepsilon,\varepsilon)$.
		\item Compactness within $\mathcal S$: any bounded closed subsets of $\mathcal S$ are compact.
	\end{enumerate}
	Then, the following hold:
	\begin{enumerate}[label=$(\alph*)$]
		\item Existence of a continuous extension of  $\mathcal R^+$: There are continuous extensions  of the parametrization in the sense that  $ (\Psi,\Phi):[0,\infty)\to \mathbb R\times X$ such that  
		\begin{equation*}
			\mathcal R^+ \subset \mathfrak R \bydef \left \{  (\Psi(s),\Phi(s)) \mid s\in [0,\infty)\right\} \subset S\cap U.
		\end{equation*}
		\item The set $\{  s\geqslant 0 : \ker (\partial_x \mathcal F[(\Psi(s),\Phi(s)]) \neq \{ 0 \}\}$ has no accumulation points.
		\item At each point, $\mathfrak R$ has a local analytic re-parametrization.
		\item One of the following scenarios occurs:
			\begin{enumerate}[label=(\arabic*)]
				\item Blow up:
					\begin{equation*}
					\lim_{s\to \infty}	\norm{(\Psi(s),\Phi(s))}=\infty.
					\end{equation*}
				\item  $(\Psi(s),\Phi(s))$ approaches the boundaries of $U$ as $s$ tends to $\infty$.
				\item $\mathfrak R$ is a closed loop in the sense that 
					\begin{equation*}
						\mathfrak R= \left \{  (\Psi(s),\Phi(s)) \mid s\in [0,T]\right\} ,
					\end{equation*}
					for some $T>0$, and $ (\Psi(T),\Phi(T)) = (\lambda_0,x_0)$. Here, we may assume that $T$ denotes the smallest period of the map $s\mapsto (\Psi(s),\Phi(s)).$
			\end{enumerate}
		\item Sufficient condition for loops to occur: If one has that 
			\begin{equation*}
				(\Psi(s_1),\Phi(s_1))= (\Psi(s_2),\Phi(s_2)), \quad \text{ and } \quad \ker (\partial_x \mathcal F[(\Psi(s_1),\Phi(s_1)]) =\{ 0 \},
			\end{equation*}
			for some $s_1\neq s_2$, then the loop scenario occurs, and $|s_1-s_2|$ is an integer multiple of $T$. 
	\end{enumerate}
\end{theorem}
It turns out that there is a general setting in which   loops can be ruled out from the general picture of Theorem \ref{thm:global:CR}. This particularly involves the case of bifurcations in cones. 
 This is the content of the next theorem, the statement of which is taken from \cite[Theorem 9.2.2]{BJ03}. 
 
 For convenience, we recall  that a set $\mathcal K$ from a Banach space $X$ is called a cone centered at $x_0\in X$ if it holds that 
\begin{equation*}
	\alpha(x-x_0)\in \mathcal K, \quad \text{ for all } x\in \mathcal K \text{ and any } \alpha \geqslant 0.
\end{equation*}
\begin{theorem}[Global analytic bifurcation in cones]\label{thm:global:cone}
	In addition to the hypotheses of Theorem $\ref{thm:global:CR},$ we assume further that 
	\begin{enumerate}[label=$(\roman*)$]
		\item There is a cone $\mathcal K$ in $X$, centered at $x_0$, such that 
		\begin{equation*}
			\mathcal R^+\subset \mathbb R \times \mathcal K,
		\end{equation*}
		for $\varepsilon$ being taken sufficiently small
		\item If $\lambda \in \mathbb R$ and $\xi \in \ker (\partial_x \mathcal F[(\lambda,x_0)])$, then one necessarily has that 
		\begin{equation*}
			\xi = \alpha \xi_0, \quad \text{ and } \quad \lambda=\lambda_0,
		\end{equation*}
		for some $\alpha \geqslant 0 $.
		\item Each point of $\mathfrak R \cap \mathcal T \cap \left( \mathbb R \times \mathcal K\right)$ is an interior point of $\mathcal T \cap \left( \mathbb R \times \mathcal K\right)$ in $\mathcal S$.
	\end{enumerate}
	Then, it holds that $\Phi(s)\in \mathcal K\setminus \{0\}$ for all $s>0$, and the loop scenario from Theorem $\ref{thm:global:CR}$ does not occur.
\end{theorem}

\section*{Acknowledgements} 
T. Hmidi has been supported by Tamkeen under the NYU Abu Dhabi Research Institute grant.  F. Rousset is supported by the BOURGEONS project, grant ANR-23-CE40-0014-01 of the French National Research Agency (ANR).

\bibliographystyle{plain} 
\bibliography{singular}

\end{document}